\documentclass[11pt]{amsart}
\usepackage{
xcolor,
cite
}
 \usepackage[colorlinks=true]{hyperref}
 \hypersetup{urlcolor=blue, citecolor=red}

\usepackage{amssymb}
\usepackage{amsmath}
\usepackage{amscd}
\usepackage{curves}
\usepackage{epsfig}
\usepackage{bbm}
\usepackage{geometry}
\usepackage{graphicx}
\usepackage{pstricks}
\usepackage{pstricks-add}
\usepackage{pst-grad}
\usepackage{pst-plot}
\usepackage{verbatim}
\usepackage{latexsym}
\usepackage{eucal}
\usepackage{amsfonts}
\usepackage{enumerate}
\numberwithin{equation}{section}
\newtheorem{theorem}{Theorem}[section]
\newtheorem{lemma}[theorem]{Lemma}
\newtheorem{prop}[theorem]{Proposition}
\newtheorem{corollary}[theorem]{Corollary}

\DeclareSymbolFont{boldoperators}{OT1}{cmr}{bx}{n}
\SetSymbolFont{boldoperators}{bold}{OT1}{cmr}{bx}{n}
\edef\bar{\unexpanded{\protect\mathaccentV{bar}}\number\symboldoperators16}

\def \dovet {\frac{\del}{2}}

\def \bpf {\begin{proof}}
\def \epf {\end{proof}}
\def \beq {\begin{equation}}
\def \eeq {\end{equation}}
\def \bsp{\begin{split}}
\def \esp{\end{split}}

\def \wt {\widetilde}

\def \det {\operatorname{det}}

\def \id {{\operatorname{I}}}

\def \vsig {{\varsigma}}

\def \mca {{\mathcal A}}
\def \mcb {{\mathcal B}}

\def \mcd {{\mathcal D}}
\def \mce {{\mathcal E}}

\def \mcg {{\mathcal G}}
\def \mch {{\mathcal H}}

\def \mcl {{\mathcal L}}
\def \mcm {{\mathcal M}}
\def \mcw {{\mathcal W}}

\def \mcp {{\mathcal P}}
\def \mcq {{\mathcal Q}}
\def \mcr {{\mathcal R}}

\def \mcu {{\mathcal U}}
\def \mcv {{\mathcal V}}

\def \mcx {{\mathcal X}}
\def \mcy {{\mathcal Y}}

\def \mbc {{\mathbb C}}

\def \mh {{\mathbb H}}
\def \mb {{\mathbb B}}

\def \mr {{\mathbb R}}
\def \mn {{\mathbb N}}
\def \ms {{\mathbb S}}

\def\ha {\frac{1}{2}}
\def \tha  {\frac32}

\def \oq {\frac{1}{4}}

\def \nsq {\frac{n^2}{4}}

\def \ka {\kappa}

\def \Tau {{\mathcal{T}}}

\def \intx {\mathring{X}}

\def \vol {\operatorname{vol}}

\def \loc {\operatorname{loc}}

\def \det {\operatorname{det}}

\def \im {\operatorname{Im}}
\def \re {\operatorname{Re}}

\def \defeq {\stackrel{\operatorname{def}}{=}}

\def \mcn {\mathcal{N}}

\def \msn {{\mathbb S}^n}

\def \ga {{\gamma}}
\def \eps {\varepsilon}   
   
\def \Ups {\Upsilon}   
   
\def \vphi {\varphi}
\def \vrho {\varrho}

 \def \rhz {R_{H_0}}
\def \rh {R_{H}}

\def \la {\lambda}   
   
\def \lan {\langle}   
\def \ran {\rangle}   
\def \del {\delta}   

\def \p {\partial}

\def \novt {\frac{n}{2}}

\def \fnf{\frac{n^2}{4}}
\def \beqq {\begin{equation}}
\def \eeqq {\end{equation}}

\numberwithin{equation}{section}

\begin{document}
\title[Eigenvalues of Schr\"odinger operators on hyperbolic manifolds]{The Distribution of Negative Eigenvalues of Schr\"odinger Operators on Asymptotically Hyperbolic Manifolds}
\author{Ant\^onio S\'a Barreto}
\address{Ant\^onio S\'a Barreto\newline \indent Department of Mathematics Purdue University \newline
\indent 150 North University Street, West Lafayette Indiana,  47907, USA}
\email{sabarre@purdue.edu}
\author{Yiran Wang}
\address{Yiran Wang\newline \indent Department of Mathematics Emory University \newline \indent 400 Dowman Drive
Atlanta, GA 30322}
\email{yiran.wang@emory.edu}
\thanks{ The first author is partly supported by the Simons Foundation grant \#848410. The second author is partly supported by National Science Foundation under grant DMS-2205266. }

\begin{abstract} 
We study the asymptotic behavior of the counting function of negative eigenvalues of  Schr\"odinger operators with real valued potentials on asymptotically hyperbolic manifolds. We establish conditions on the potential that determine if there are finitely or infinitely many negative eigenvalues.  In the latter case,  they may only accumulate at zero and we obtain the asymptotic behavior of the counting function of eigenvalues in an interval $(-\infty, -E)$ as $E\rightarrow 0.$
\end{abstract}

\maketitle


\section{Introduction}

We are concerned with the following type of problem:  Let $(X,g)$ be a non-compact complete $C^\infty$ Riemannian manifold and let  $\Delta_g$ be its (positive) Laplacian. Suppose that $V$ is a real valued potential such that $V<0$ near infinity and the corresponding Schr\"odinger operator  $H=\Delta_g+V$ is self-adjoint.  Furthermore, suppose its  point spectrum $\sigma_{p}(H)\subset (-E_0,0)$ and the eigenvalues only accumulate at zero.  The problem is to find conditions on $V$ which determine whether the point spectrum is finite or infinite and if it is infinite, find the asymptotic behavior  of the number of eigenvalues (counted with multiplicity) in an interval $(-E_0,-E)$ as $E\downarrow 0.$

In the Euclidean case, it has been shown, see for example \cite{RS4},  that if $V$ is bounded and if that near infinity $V(z)\leq - C |z|^{-2+\del},$ $\del > 0,$ then $H$ has infinitely many eigenvalues, while  if $V\geq -C|z|^{-2-\del},$ $\del>0$ there are finitely many eigenvalues. The threshold decay of $V(z)$ for $H$  to have finitely or infinitely many eigenvalues is therefore $V(z)\sim  F(\omega)|r|^{-2},$ $r=|z|,$  $z=r\omega,$ $\omega\in \ms^{n-1}.$   Moreover, an application of Hardy's inequality shows that if  $V(z)\sim -c r^{2},$ there are finitely many eigenvalues when $c<\frac{(n-1)^2}{4}$ and infinitely many if $c>\frac{(n-1)^2}{4}.$  Precise results on the asymptotics of the counting function of eigenvalues  as $E\rightarrow 0,$ in the case where $V(z)= r^{-2}( F(\omega) + \mce(r,\omega))$  with $\mce(r,\omega)=o((\log r)^{-1-\eps})$ as $r\rightarrow \infty$ were obtained by Kirsch and Simon \cite{KS} and Hassell and Marshal \cite{HM}.  We are not aware of similar results for asymptotically Euclidean manifolds.

 While this problem has been well studied in the Euclidean space,  it seems that it has not been studied as much in hyperbolic space.  We will  work in the class of asymptotically hyperbolic manifolds in the sense of  \cite{FefGra,MM}, for which the hyperbolic space serves as a model.  Akutagawa and Kumura \cite{AkuKum} have established bounds on the potential, similar to those in Euclidean space, which determine if discrete spectrum is finite or infinite (as in the first two items of Corollary \ref{low-up} below), but they do not  establish bounds on the counting function of negative eigenvalues and do not consider the cases of critical decay, as in our Theorems  \ref{main1} and \ref{doubleT}.   We should also mention the work of Mazzeo and McOwen \cite{MMC}.  While the problems considered in \cite{MMC}  are somewhat the opposite of the ones we study here, there are some similarities. 
  
  The Poincar\'e model of hyperbolic space $\mh^{n+1}$ with constant curvature $-1$ is given by the Euclidean ball of radius one
 \beq \label{HYP}
 \mb^{n+1}=\{z\in \mr^{n+1}: |z|<1\} \text{ equipped with the metric } g_0(z)=\frac{4 dz^2}{(1-|z|^2)^2}.
 \eeq

This is the interior of a  $C^\infty$ manifold equipped with a metric which is singular at its boundary.  The function  $\vsig(z)= 1-|z|^2 \in C^\infty(\overline{\mb^{n+1}}),$ satisfies $\vsig(z)\geq 0,$ $\vsig^{-1}(0)= \msn= \p \mb^{n+1}$ and $d\vsig(z)\not=0$ if $|z|=1,$ and so $\vsig$ is called a defining function of $\p \mb^{n+1},$  and moreover $\vsig^2 g_0= 4dz^2$
is a $C^\infty$ Riemannian metric on the closure $\overline{ \mb^{n+1}}.$  Following \cite{FefGra,MM,Megs}, one can extend this notion to any $C^\infty$ Riemannian manifold with boundary.

 Throughout this paper,  $\intx$ will denote the interior of a
$C^\infty$ compact manifold $X$ with boundary $\p X$  of dimension $n+1.$    We say that $\varsigma\in C^\infty(X)$ is a defining function of $\p X,$ or a boundary defining function, if  $\varsigma > 0$ in $\intx,$ $\varsigma = 0$ at $\p X$ and $d\varsigma \neq 0$ at $\p X$.  We assume $\intx$ is equipped with a $C^\infty$ Riemannian metric $g$ such that $G = \varsigma^2 g$ is non-degenerate at $\p X$  and so $(\bar X,G)$ is a $C^\infty$ compact Riemannian manifold with boundary.   According to \cite{Ma1}, the sectional curvature of $(\intx,g)$ converges to $-|d\varsigma|_G$  along any curve that goes towards $\p X.$    The manifold $(\intx, g)$ is called an asymptotically hyperbolic manifold, or AHM, if  $|d\varsigma|_G = 1$ at $\p X.$ One might relax this assumption if  $X$ has more than one boundary component  and  $\p X= Y_1 \sqcup Y_2 \ldots \sqcup Y_N$ and $|d\varsigma|_G = \ka_j$ at $Y_j,$ $\ka_j$ is a constant.

The  hyperbolic space  serves as a model for this class of manifolds, and its quotients by certain discrete groups of fractional linear transformations having a geometrically finite fundamental domain without cusps at infinity are also examples of such manifolds, see \cite{Agmon1,MM,Perry1,Perry2}.    In fact our results apply for manifolds with more than one boundary component and with different (constant) asymptotic curvatures at each end. One important example from mathematical physics  where this occurs is the stationary model of the de Sitter-Schwarzschild model of black holes discussed in \cite{SaZw}. In this case the manifold is $\intx= (a, b) \times \ms^n,$ and the asymptotic curvatures are different on both ends.

  If $(\intx,g)$ is an AHM and $\varsigma\in C^\infty(X)$ is a boundary defining function, then by assumption $G=\vsig^2 g$ is non-degenerate up to $\p X,$ but  the metric $\vsig^2 G|_{\p X}=h_0$  obviously depends on the choice of $\vsig.$  In fact, given any two boundary defining functions $\vsig$ and $\tilde \vsig,$ we must have $\vsig=a(z)\tilde\vsig,$ with $a>0.$ If $\tilde G={\tilde \vsig}^2 g,$  then $G=a^2 \tilde G,$  and in particular $G|_{\p X}=(a^2 \tilde G)|_{\p X},$ and hence  $\vsig^2g$ determines a conformal class of metrics at $\p X.$   As shown in \cite{Gr,JosSaB}, given a representative $h_0$ of the class $[\vsig^2 g|_{\p X}]$ there exists  $\mathfrak{e}>0$ and a unique boundary defining function $x$ defined on a collar neighborhood $U$ of $\p X$  and a map
$\Psi: [0,\mathfrak{e}) \times \p X \longrightarrow U$ such that 
\beqq\label{prod}
\Psi^* g = \frac{dx^2}{x^2} + \frac{h(x)}{x^2}, \;\ h(0)=h_0,
\eeqq
where  $h(x)$ is a $C^\infty$ one-parameter family of  Riemannian metrics on $\p X.$   Of course, $x$ can be extended (non-uniquely) from the collar neighborhood $U$ to a boundary defining function $x\in C^\infty(X).$

 For example,  in the case of the hyperbolic space  \eqref{HYP}, the geodesic distance with respect to the origin is given by 
 \[
 r= \log\left(\frac{1+|z|}{1-|z|}\right),
 \]
and using polar coordinates $(r,\theta),$ $\theta=\frac{z}{|z|},$  with respect to this distance $r,$  the metric $g_0$  
is given by
 \[
 g_0= dr^2 + (\sinh r)^{2} d\theta^2,
 \]
 where $d\theta^2$ is the standard metric on the sphere.  If we set $x=e^{- r},$ then 
\beq\label{prod10}
g_0= \frac{dx^2}{x^2} + \frac{(1-x^2)^2}{4} \frac{d\theta^2}{x^2}.
\eeq
While $x$ is not smooth on $X$ because $|z|$ is not $C^\infty$ at $\{z=0\},$  it is $C^\infty$ near $\p \mb^{n+1},$ and one can modify it in the interior of $X$ to satisfy the definition of a boundary defining function and still keep \eqref{prod10} near $\p \mb^{n+1}.$

Let $\Delta_{g}$ denote the  (positive) Laplace operator on an AHM $(\intx, g).$   We know from \cite{Ma1,Ma2}, 
that the spectrum of $\Delta_g,$ denoted by $\sigma(\Delta_g),$ satisfies $\sigma(\Delta_g)= \sigma_{pp}(\Delta_g) \cup \sigma_{ac}(\Delta_g),$ where $\sigma_{pp}(\Delta_g)$  consists of a finite number of eigenvalues in $(0, \nsq)$ with finite multiplicity and $\sigma_{ac}=[\nsq,\infty)$ is the absolutely continuous spectrum.  There are no eigenvalues in $[\nsq,\infty),$ see for example \cite{bouclet,Ma2}.   We shall work with 
$\Delta_g-\nsq$ which has continuous spectrum  $[0,\infty).$  Let $V \in L^\infty(X)$  be real valued and such that $V(z)<0$ near $\p X$ and $V(z)\rightarrow 0$ as $z\rightarrow \p X.$ We shall denote
\beq\label{defoP}
H_0=\Delta_g-\nsq \text{ and }  H=\Delta_g-\nsq+V.
\eeq
We will show that  under the assumptions on the rate of decay of the potential $V(z)$  as $z\rightarrow \p X,$ the point spectrum of $H$
consists of eigenvalues of finite multiplicity  contained in some interval $[-E_0,0),$ which only possibly accumulate at zero.  Its essential spectrum $\sigma_{ess}(H)=[0,\infty)$ and it has no embedded eigenvalues, including the bottom of $\sigma_{ess}(H).$     We want to count negative eigenvalues of $H$ and other operators, and so for an operator $T$ and for $E\geq 0,$  we define the counting function 
\beq\label{counting}
N_E(T)= \# \{ \mu \in (-\infty,-E) \cap \sigma_{pp}(T),  E\geq 0 \text{ counted with multiplicity} \}.
\eeq

In  coordinates for which \eqref{prod} is valid, the Laplacian with respect to $g$ is given by
\beq\label{lap1}
  \Delta_g= -(x\p_x)^2-nx\p_x - x^2 A(x,y) \p_x+ x^2 \Delta_{h(x)},
\eeq
where $\Delta_{h(x)}$ is the Laplacian with respect to the metric $h(x)$ on $\p X,$ and $A=\ha \p_x \log |h|,$ where $|h|$ is the volume element of the metric $h.$   It is convenient to define $\rho\defeq -\log x,$  and so
\beq\label{deflap}
 \Delta_g= -\p_\rho^2-n\p_\rho - e^{-\rho} \mca(\rho,y) \p_\rho+ e^{-2\rho} \Delta_{\wt h(\rho)},
 \eeq
where $\mca(\rho,y)=A(e^{-\rho},y)$ and $\wt h(\rho)=h(e^{-\rho}).$

Throughout the paper,  $(\intx, g)$ will denote a $n+ 1$ dimensional asymptotically hyperbolic manifold.  Its  closure $X$ is a $C^\infty$ manifold with boundary,  and $x \in C^\infty(X)$ is a boundary defining function such that \eqref{prod} holds for $x\in [0,\mathfrak{e}).$  We define  $\rho\defeq -\log x.$  We assume that  $V\in L^\infty(X)$ {\it is real valued,} $H_0,$ $H$  and $N_E(H)$ will be defined as in \eqref{defoP} and \eqref{counting} respectively.

\begin{theorem}\label{main2} Suppose that  in some interval $\rho\in(\rho_0,\infty),$ there exists a constant $c>0$ such that
\beq\label{hyp-pot}
V(e^{-\rho},y) = -c\rho^{-2+\del} + O(\rho^{-2+\del}(\log \rho)^{-\eps}), \text{ for } \eps>0 \text{ and }  \del<2,  \text{ as } \rho\rightarrow \infty.
\eeq

If $\del<0,$ then $N_0(H)<\infty,$ but if $\del\in (0,2),$ then $N_0(H)=\infty$ and 
\beq\label{subc}
\log \log N_E(H)=  \frac{1}{2-\del} \log E^{-1}+ O(1), \text{ as } E\rightarrow 0.
\eeq
\end{theorem}

Notice that if $\wt x=e^{\vphi(x,y)}  x,$ $\vphi\in C^\infty,$  is another boundary defining function, then
\[
\wt \rho = -\log \wt x = -\vphi- \log x = \rho+ O(1), \text{ as } \rho\rightarrow \infty,
\]
so \eqref{hyp-pot} does not depend on the choice of  $x.$

In the threshold case $\del=0,$ case we have  the following

\begin{theorem}\label{main1}   Suppose that  in some interval $\rho\in(\rho_0,\infty),$
\beq\label{hyp-potTC}
V(e^{-\rho},y) =-c\rho^{-2}+ o(\rho^{-2}(\log \rho)^{-\eps}), \; c>0, \;\ \eps>0, \text{ as } \rho\rightarrow \infty.
\eeq
 If $c<\oq,$  then   $N_0(H)<\infty,$ but if $c > \oq,$ then  $N_0(H)=\infty$ and 
\beq \label{bdhs1}
\log \log N_E(H)= -\ha \log E+ O(1) \text{ as } E\rightarrow 0.
\eeq
\end{theorem}
Our proofs in fact give somewhat more precise  upper and lower bounds for $N_E(H),$ and  \eqref{subc} and \eqref{bdhs1} are used to unify these bounds and provide the asymptotic behavior of iterated logarithms of $N_E(H).$ 
The methods we use do not allow us to treat the case where $c$ is a function of $y.$ However, we can use these results to prove

\begin{corollary}\label{low-up}   Suppose that  in some interval $\rho\in(\rho_0,\infty),$
\beq\label{uplow-del}
-c_1\rho^{-2+\del} \leq V(e^{-\rho},y) \leq -c_2\rho ^{-2+\del }, 
\eeq
then we can say that
\begin{enumerate}[1.]
\item If $\del<0,$ then $N_0(H)<\infty.$
\item If $\del \in (0,2),$ then  $N_0(H)=\infty$ and \eqref{subc} holds.
\item If  $\del=0$ and $c_1<\oq,$  then $N_0(H)<\infty.$  
\item If  $\del=0$ and  $c_2 > 1/4$, then  $N_0(H)=\infty$ and  \eqref{bdhs1} holds.
\end{enumerate}
\end{corollary}

We can say more in the threshold case $c=\oq.$   For $\rho$ large enough, we define
\beq\label{deflj}
\log_{(j)} \rho= \log \log \ldots \log \rho, \text{ j times.}
\eeq

\begin{theorem}\label{doubleT}  Suppose that  in some interval $\rho\in(\rho_0,\infty),$ there exists a constant $c_1>0$ such that 

\beq\label{defV1}
V(e^{-\rho},y)= -\oq \rho^{-2}-c_1 \rho^{-2}(\log \rho)^{-2}+  O\left( \rho^{-2} (\log \rho)^{-2} (\log \rho)^{-\eps}\right), \; \eps>0.
\eeq
If $c_1<\oq,$ then $N_0(H)<\infty$ and  if $c_1>\oq,$ $N_0(H)=\infty$ and 
\beq \label{bdhs2}
\log_{(3)} N_E(H)= \log_{(2)} \left(E^{-1}\right)+ O(1) \text{ as } E\rightarrow 0.
\eeq

In fact this process keeps going indefinitely and the result holds if for some $\rho_0$ large, and $\rho\in [\rho_0,\infty),$ the potential has an expansion of the form 
\beq\label{def-GN}
\begin{split}
  V(e^{-\rho},y) & \ = V_0(\rho)+  O\left( \mcg_{N}(\rho) (\log \rho)^{-\eps}\right), \;\ \eps>0, \text{ where for some } N\in \mn, \\
 V_0(\rho) & \ = -\oq \rho^{-2}-\oq \sum_{j=1}^{N-1} \mcg_{j}(\rho) + c_N \mcg_{N}(\rho), \;\ c_N \text{ is a constant} \text{ and  } \\
& \mcg_{(j)}(\rho)  =\rho^{-2} (\log \rho)^{-2} (\log \log \rho )^{-2}\ldots (\log_{(j)} \rho)^{-2}.
\end{split}
\eeq
 The existence of infinitely many eigenvalues depends on whether $c_N<\oq,$ or $c_N>\oq.$  If $c_N<\oq$ there are only  finitely many eigenvalues, but if $c_N>\oq,$ 
\beq \label{bdhs4}
\log_{(N+2)}  N_E(H)=  \log_{(N+1)} (E^{-1})+ O(1) \text{ as } E\rightarrow 0.
\eeq
\end{theorem}

Notice that one cannot hope to take $N=\infty$ in the definition of $V_0(\rho)$ because the denominators will be equal to zero at points of the form $\rho= e^{e^{e^{\ldots}}}.$ As in the case of Theorems \ref{main2} and \ref{main1}, our proofs actually give better upper and lower bounds for $N_E(H)$ and this formulation is used to unify these bounds.

As we have already mentioned, the metric $g$ induces a conformal structure at $\p X$ and this is reflected in the choice of the boundary defining function $x.$   Since $N_E(H)$ does not depend on this choice, its asymptotic behavior  in principle could  reveal some  invariants of the conformal structure of the metric induced by $g$ at $\p X.$  However,  our methods are not refined enough to detect that.  This dependence will not affect the main term of the asymptotic behavior of $N_E(H)$ and the contributions of the boundary structure will be hidden among  the  terms  of the $O(1)$ part of the estimates above and these are very hard to tract.

\subsection{The Strategy of the Proofs }   The methods used in the proof of Theorems \ref{main2}, \ref{main1} and \ref{doubleT} are  Dirichlet-Neumann bracketing and the Sturm oscillation theorem,  which are standard for this type problems. 

 For $x$ as in \eqref{prod} and $ x_0\in [0,\mathfrak{e}),$  let
\beq\label{neigh}
X_\infty = \{ z \in X : x(z) \leq x_0\}, \ \ X_0 = \{ z\in  X : x(z) \geq x_0\}.
\eeq
So $(X_0,g)$ and $(X_\infty, x^2g)$ are $C^\infty$ compact Riemannian manifolds with boundary. We will define $\Upsilon^\bullet_{0},$ and $\Upsilon^\bullet_{\infty},$ to  be the restrictions of the operator $H$ to $X_0$  and $X_\infty$ with Dirichlet ($\bullet=D$) and Neumann ($\bullet=N$) boundary conditions.   Since $(X_0,g)$ is a compact $C^\infty$ Riemannian manifold with boundary, it is well known, see for example \cite{RS4}, that 
\beq\label{spec-X0}
\begin{split}
& \sigma(\Upsilon_0^D)= \{ \wt \la_1 < \wt \la_2 \leq \wt  \la_3  \ldots \},  \;\ \wt \la_j\in \mr,  \;\ \wt \la_j\rightarrow \infty, \\
& \sigma(\Upsilon_0^N)= \{ \wt \mu_1 < \wt \mu_2 \leq \wt \mu_3  \ldots \},  \;\ \wt \mu_j\in \mr, \;\  \wt \mu_j\rightarrow \infty.
\end{split}
\eeq

We will show that $ \sigma_{ess}(\Upsilon^\bullet_\infty)=[0,\infty),$  $\bullet=D,N,$  with no embedded eigenvalues, and their point spectra satisfy
\beq\label{spec-XI}
\begin{split}
& \sigma_{pp}(\Upsilon^D_\infty)= \{ \la_1 < \la_2 \leq  \la_3 \ldots \}, \;\  \la_j<0, \text{ is finite or }  \la_j\rightarrow 0,\\
& \sigma_{pp}(\Upsilon^N_\infty)=  \{ \mu_1 <  \mu_2 \leq  \mu_3 \ldots \}, \;\   \mu_j<0, \text{ is finite or }   \mu_j\rightarrow 0.
\end{split}
\eeq

 Following Chapter XIII of \cite{RS4}, we will show that
\[
\Upsilon_0^N\oplus \Upsilon_{\infty}^N \leq H \leq \Upsilon_{0}^D \oplus \Upsilon_{\infty}^D,
\]
 and it follows that for any $E <0,$ 
\begin{equation}
N_E(\Upsilon_{0}^N) + N_E(\Upsilon_{\infty}^N)\leq  N_E(H) \leq N_E(\Upsilon_{0}^D) + N_E(\Upsilon_{\infty}^D),\label{eqest}
\end{equation}
where $N_E(T)$ is the counting function defined in \eqref{counting} for the operators $T=\Upsilon_0^\bullet$ and $\Upsilon_\infty^\bullet$ instead of $H.$
But  it follows from \eqref{spec-X0} that there exists $N^{\#}>0$ such that  for any $E<0,$  $N_E(\Upsilon_{0}^N) <N^{\#}$ and $ N_E(\Upsilon_{0}^D)<N^{\#}.$ We will show that  if $V$ satisfies the hypotheses of either one of the Theorems \ref{main2}, \ref{main1} and \ref{doubleT},  then for $E<0$, both $N_E(\Upsilon_\infty^\bullet),$ $\bullet=N,D,$  have either finitely or infinitely many eigenvalues.  In case both have infinitely many eigenvalues,  the corresponding counting function of their eigenvalues  have the same asymptotic behavior as $E\searrow 0,$ and therefore it gives the asymptotic behavior of $N_E(H).$


\section{The Spectrum of $H$ }

  For the lack of  suitable references, we  will briefly discuss some properties of the spectrum of $H.$  First we recall some results about of the spectrum of $H_0$ from from \cite{MM,Ma2}.  Let  $x$ be a boundary defining function for which \eqref{prod} holds in a collar neighborhood of $\p X.$  In these coordinates, the Laplacian $\Delta_g$ is given by \eqref{deflap}, so  $\Delta_g$ is a zero differential operator in the sense of  \cite{MM}, and  we define the zero-Sobolev spaces of order k as in \cite{MM}:  Let  $\mcv(\p X)$ denote the Lie algebra of $C^\infty$ vector fields on $X$ which are equal to zero at $\p X.$  In coordinates $(x,y)$  these vector fields are spanned   by $\{x\p_x, x\p_{y_j}, \; 1\leq j \leq n\}$ over the $C^\infty$ functions.  Let
\beq\label{0-sob}
\mch_0^k(X)=\{ u \in L^2(X):   \;\ W_1 W_2\ldots W_m u \in L^2(X), \;\   W_j \in \mcv(\p X), \;\ m \leq k\}.
\eeq
 We know from  \cite{MM} that $\Delta_g$ with domain $\mch_0^2(X)\subset L^2(X)$ is a self-adjoint operator,   we also know from \cite{bouclet,MM,Ma1} that  its spectrum  consists of an absolutely continuous part $\sigma_{ac}(\Delta_g)=[\fnf, \infty)$ and finitely many eigenvalues in the point spectrum $\sigma_{pp}(\Delta_g)\subset (0, \fnf).$   There are no eigenvalues in $[0,\infty),$ \cite{bouclet,Ma2}. As above, we set $H_0= \Delta_g-\nsq,$ and so  the resolvent
\begin{gather*}
\rhz(\la)=(H_0-\la)^{-1}: L^2(X) \longmapsto \mch_0^2(X), \\
\text{ provided } \la \in \mbc \setminus ([0,\infty) \cup \{ \la_1, \la_2,\ldots, \la_N\}),
\end{gather*}
where $\la_j \in (-\nsq, 0)$ is an eigenvalue of finite multiplicity of $H_0 .$ By definition, the resolvent set of $H_0$ is
 \beq\label{res-set}
 \rho(H_0)=   \ \mbc \setminus ([0,\infty) \cup \{\la_1, \ldots \la_N\}), \;\  \la_j \text{ is an eigenvalue of } H_0.
  \eeq
  
  To analyze the spectrum of $H,$ we begin by observing that for $\la \in \rho(H_0),$
  \beq\label{res-id}
(\Delta_g+V-\nsq-\la) \rhz(\la)= \id+ V \rhz(\la), 
\eeq

Since
\[
\rhz(\la): L^2(X) \longrightarrow \mch_0^2(X) \text{ is a bounded operator for } \la \in \rho(H_0),
\]
then  if $\chi_j \in C_0^\infty(\intx),$ $\chi_j(z)=1$ in the region $x(z)>\frac{1}{j}$ and $\chi_j(z)=0$ if $x(z)<\frac{1}{j+1},$   it follows that
\[
\chi_j(z) V(z) \rhz(\la): L^2(X) \longrightarrow H_c^2(\intx){\hookrightarrow} L_c^2(\intx),
\]
where the subindex $c$ indicates compact support. Notice that since supports are compact we can use either $\mch_0^2(X)$ or the standard Sobolev space $H^2(X).$ It  follows from Rellich's embedding Theorem that for fixed $j,$
\[
\chi_j(z) V(z) \rhz(\la): L^2(X) \longrightarrow L_c^2(\intx)
\]
is a compact operator.  Since  $V\in L^\infty(X)$ and $V(z)\rightarrow 0$ as $z\rightarrow \p X,$ it follows that 
\begin{gather*}
||\chi_j(z) V(z) \rhz(\la)- V(z) \rhz(\la)||_{\mcl(L^2(X))}\rightarrow 0 \text{ as } j\rightarrow \infty
\end{gather*}
in the operator norm, and so we conclude that $V \rhz(\la): L^2(X) \longrightarrow L^2(X)$ is a compact operator, provided $\la \in \rho(H_0).$  But we also know that  for $\im\la<<0,$ the operator norm of $\rhz(\la)$ is less than or equal to 
 $\frac{1}{\im(\la)},$ see for example Theorem VI.8 of \cite{RS1}, and therefore $(\id+ V \rhz(\la))^{-1}$ is bounded for $\im \la<<0,$ and $|\re \la|>1.$   Then Fredholm Theorem, see for example Theorem VI.14 of \cite{RS1}, guarantees that with the exception of a countable set of points, which are poles of $\rh(\la),$
\begin{gather*}
\rh(\la)=\rhz(\la)(\id+ V \rhz(\la))^{-1} \text{ for } \la \in \rho(H_0).
\end{gather*}
Moreover the poles of  $\rh(\la)$  in $\rho(H_0)$  consist of a countable set $\{\mu_j, \;\ j \in \mn\}\subset (-\infty,0)$ such that $\mu_j$ are eigenvalues of $H$ with finite multiplicity. This set is either finite, or infinite.  If there are infinitely may eigenvalues, they accumulate  only at zero. 

Finally, since $V R_{H_0}(\la)$ is compact for  $\la\in \rho(H_0),$ it follows that the operator 
 \begin{gather*}
 V: L^2(X) \longmapsto L^2(X)\\
 f \longmapsto V(z) f
 \end{gather*}
 is relatively compact with respect to $H_0$ and it follows from Weyl's Theorem, see Theorem 14.6 of \cite{HS} that $\sigma_{ess}(H)=\sigma(H)\setminus \sigma_{pp}(H)=\sigma_{ess}(H_0)=[0,\infty).$  Therefore we have
\begin{theorem}\label{eigenvT}   Let $(\intx,g)$ be an AHM, let  $V\in L^\infty(X)$ be real valued and suppose that $V(z)\rightarrow 0$ as $z\rightarrow \p X.$ Then
$\sigma_{ess}(H)=\sigma(H)\setminus \sigma_{pp}(H)=[0,\infty).$ There are no embedded eigenvalues. Moreover, the resolvent 
\begin{gather*}
\rh(\la)=(H-\la)^{-1}:L^2(X) \longrightarrow \mch_0^2(X) \text{ is bounded for } \la \in \mbc\setminus( [0,\infty)\cup  \mcd),
\end{gather*}
where $\mcd= \{\mu_1,  \mu_2, \ldots\}\subset (-\infty,0),$ with $\mu_{j+1}\geq \mu_j,$ is a bounded discrete set   (in $(-\infty,0)$) which consists of  eigenvalues of $H$ of finite multiplicity.  
\end{theorem}
 The fact that there are no embedded eigenvalues is due to Mazzeo \cite{Ma2}, see also \cite{Bor,bouclet}. 
We also have the following 
\begin{theorem}\label{specDN}  Let $(\intx,g)$ be an AHM, let $x$ be a boundary defining function such that \eqref{prod} holds.  Let $X_\infty$ be as in \eqref{neigh} and let $\Upsilon_\infty^\bullet,$ $\bullet=N,D$ denote the operator $H$ with Dirichlet or Neumann boundary conditions in $X_\infty.$ If $V(z)\rightarrow 0$ as $z\rightarrow \p X,$  according to either \eqref{hyp-pot}, \eqref{hyp-potTC} or \eqref{def-GN},  then  $\sigma_{ess}(\Upsilon_\infty^\bullet)=[0,\infty).$  Moreover, there are no embedded eigenvalues.
\end{theorem} 

The fact that $\sigma_{ess}(\Upsilon_\infty^\bullet)=[0,\infty)$ is a consequence of Theorem \ref{eigenvT} and Proposition 2.1 of \cite{Don}, see also Theorem 9.43  of \cite{Bor}.  The proofs of these results are actually  for the Dirichlet boundary conditions, but they work for Neumann conditions as well.   The results of \cite{Bor,bouclet,Ma2} also guarantee that there are no embedded eigenvalues in these cases. The point is to show that if there were eigenfunctions in $L^2,$ they would decay exponentially and a Carleman estimate shows that they are actually equal to zero. The argument takes place in a neighborhood of $\p X$ and thus also works for $\Upsilon_\infty^\bullet.$  


\section{Dirichlet-Neumann Bracketing}\label{DNB}

One can view the operator $H$ as the unique self-adjoint operator on $L^2(X)$ whose quadratic form is the closure of
\beq\label{qformX}
\begin{split}
Q(\vphi,\psi)= \langle \nabla_{g} \phi, \nabla_{g} \psi\rangle_{{}_{L^2_g(X)}} +&  \langle (V-\nsq) \phi, \psi\rangle_{{}_{L^2_g(X)}} = \int_{X} g^{ij}\p_i\phi \p_j\bar\psi d\vol_g + \int_{X} (V-\nsq)\phi\bar \psi d\vol_g, \\
& \text{ with }  \phi,\psi  \in C_0^\infty(\intx).
\end{split}
\eeq
The domain of the quadratic form $Q$ is $\mch_0^1(X)\times \mch_0^1(X),$ defined in \eqref{0-sob}.

Let $x$ be a boundary defining function such that \eqref{prod} holds and let $X_0$ and $X_\infty$ be as defined in \eqref{neigh}.   We consider the quadratic forms to be the closure of 
\beq\label{qform}
\begin{split}
 Q^D(X_0)(\vphi,\psi)= Q(\vphi,\psi) &  \text{ with } \vphi, \psi \in C_0^\infty(\intx_0), \\
 Q^N(X_0)(\vphi,\psi)=  Q(\vphi,\psi) &  \text{ with } \vphi, \psi \in C^\infty(X_0), \;\ \p_x \vphi|_{\{x=x_0\}}=\p_x \psi|_{\{x=x_0\}}=0.
\end{split}
\eeq
It turns out that the domains of  these quadratic forms are
 \[
 \begin{split}
&  \mcd(Q^D(X_0))= H_0^1(X_0) \times H_0^1(X_0), \;\
 H_0^1(X_0)=   \overline{C_0^\infty(\intx_0)} \text{ with the }  H_{\loc}^1(\intx) \text{ norm},  \\
 &  \mcd(Q^N(X_0))= \bar H^1(X_0) \times \bar H^1(X_0), \;\ \bar H^1(X_0)= \{\vphi  \in L^2(X_0): \exists \ f \in H_{\loc}^1(\intx), \  f= \vphi \text{ in } X_0\}
\end{split}
\]

The self-adjoint operator operators corresponding to $Q^D(X_0)$ and $Q^N(X_0)$ are defined to be the operator $H$ respectively with Dirichlet  and Neumann boundary conditions, which we denote by $\Upsilon_0^D$ and $\Upsilon_0^N$ respectively.

Similarly, we define the  quadratic forms
\beqq\label{qform1}
\begin{split}
& Q^D(X_\infty)= Q(\vphi,\psi), \;\  \vphi,\psi \in C_0^\infty(\intx_\infty), \\
& Q^N(X_\infty)= Q(\vphi,\psi), \;\  \vphi,\psi \in C^\infty(X_\infty)\cap \mch_0^2(X_\infty) \;\ \p_x \vphi|_{\{x=x_0\}}=\p_x \psi|_{\{x=x_0\}}=0.
\end{split}
\eeqq
The domains  of their closures are 
\[
\begin{split}
& \mcd(Q^D(X_\infty))= \mcw_0^1(X_\infty) \times \mcw_0^1(X_\infty), \  \mcw_0^1(X_\infty)= \overline{ C_0^\infty(\intx_\infty) }\text{ with the }  \mch_0^1(X) \text{ norm }\\
& \mcd(Q^N(X_\infty))= \bar\mch_0^1(X_\infty) \times \bar\mch_0^1(X_\infty), \  \bar\mch_0^1(X_\infty)= \{\vphi  \in L^2(X_\infty): \exists  f \in \mch_0^1(X);  f= \vphi \text{ in } X_\infty\}.
\end{split}
\]
The corresponding self-adjoint operators are defined to be $\Upsilon_\infty^D$ and $\Upsilon_\infty^N,$ which are the operator $H$ with Dirichlet or Neumann boundray conditions on $X_\infty.$

We follow section XIII.15 of \cite{RS4} and define the direct sum of self-adjoint operators: If $A_j,$ $j=1,2,$ are self adjoint operators acting on Hilbert spaces $\mcl_j,$ $j=1,2,$ with domains $\mcd(A_j),$ $j=1,2,$ let 
\[
\mcl=\mcl_1\oplus \mcl_2 \text{ and } A_1\oplus A_2(\phi_1,\phi_2)= (A_1 \phi_1, A_2 \phi_2), \phi_j \in \mcd(A_j), \;\ j=1,2.
\]
It is proved in section XII.15 of \cite{RS4} that
\begin{enumerate}[1.]
\item $A_1\oplus A_2$ is self adjoint.
\item The associated quadratic forms satisfy $Q(A_1\oplus A_2)=Q(A_1) \oplus Q(A_2).$
\item If $N(\la,A)= \operatorname{dim} P_{(-\infty, \la)}(A),$ then
\beq\label{P1}
N(\la, A_1\oplus A_2)= N(\la,A_1) + N(\la, A_2).
\eeq
\end{enumerate}

In our case, we have four natural operators $\Upsilon^\bullet(X_0)$ and $\Upsilon^\bullet(X_\infty),$ $\bullet=D,N.$  Notice that
\[
\begin{split}
&  H_0^1(X_0) \oplus \mcw_0^1(X_\infty)  \subset \mch_0^1(X), \text{ and that for  } \;\ \vphi \in C_0^\infty(\intx_0), \;\ \psi \in C_0^\infty(\intx_\infty),  \\
 & Q^D(X_0)(\vphi,\vphi)+   Q^D(X_\infty)(\psi,\psi)  = Q(H)(\vphi,\vphi)+ Q(H)(\psi,\psi).
\end{split}
\]

On the other hand, if $\vphi\in \mch_0^1(X),$ 
\[
\vphi|_{X_0} \in \bar \mch_0^1(X_0) \text{ and } \vphi|_{X_\infty} \in \bar \mch_0^1(X_\infty),
\]
this means that
\[
\begin{split}
& \mch_0^1(X_0)  \subset \bar \mch_0^1(X) \oplus \bar \mch_0^1(X_\infty), \text{ and for } \;\ \vphi \in \mch_0^1(X), \\
 & Q^N(X_0)(\vphi,\vphi) +  Q^N(X_\infty)(\vphi,\vphi)  = Q(H)(\vphi,\vphi).
\end{split}
\]
 If $A$ and $B$ are self adjoint operators defined on a Hilbert space $\mcl$ and $Q(A)$ and $Q(B)$ are their corresponding quadratic forms with domains $\mcd(Q(A))$ and  $\mcd(Q(B))$
 \[
 Q(A)(\vphi,\vphi) \geq  M ||\phi||, \;\ \phi \in \mcd(Q(A)) \text{ and }  Q(B)(\psi,\psi)\geq  M ||\psi||, \;\ \psi \in \mcd(Q(B)),
 \]
 we say that
 \[
 A \leq B \text{ if }  \mcd(Q(B)) \subset \mcd(Q(A)) \text{ and }  Q(A)(\vphi,\vphi) \leq Q(B)(\vphi,\vphi) , \vphi \in \mcd(Q(B)).
 \]

This translates into the following
\begin{prop}  Let $X_0,$ $X_\infty,$  $\mch^D(X_\bullet)$ and $\mch^N(X_\bullet),$ $\bullet=0,\infty,$ be defined as above, then
\[
\Upsilon_0^N\oplus \Upsilon_\infty^N \leq H \leq \Upsilon_0^D \oplus \Upsilon_\infty^D.
\]
\end{prop}

It follows from \eqref{P1} that if $E>0,$
\beq\label{eigen-est1}
N_E(\Upsilon_0^N) + N_E(\Upsilon_\infty^N) \leq N_E(H) \leq  N_E(\Upsilon_0^D) + N_E(\Upsilon_\infty^D).
\eeq

Since $\Upsilon_0^D$ and $\Upsilon_0^N$ are Schr\"odinger operators with $C^\infty$ potentials on compact manifolds with boundary, their spectra satisfy \eqref{spec-X0}.  We will show  that the point spectra of both $\Upsilon_\infty^D$ and $\Upsilon_\infty^N$ are either finite or infinite.  In the latter case, we will show that both have the same asymptotic behavior as $E$ goes to zero, and so this gives the asymptotic behavior of $N_E(H)$ as $E\rightarrow 0.$

\subsection{Model Operators on $X_\infty$}  Our analysis in this section  will be restricted to a small enough collar neighborhood of  $\p X=\{x=0\}$ where \eqref{prod} holds and so $\Delta_g$ is given by \eqref{lap1}. Since $h(x)$ is a $C^\infty$ one-parameter family of tensors on $\p X,$  we may write in local coordinates, 
\beq\label{coordh}
\begin{split}
& h(x)=\sum_{j,k=1}^n h_{jk}(x,y) dy_j d y_k, \;\ h_{jk}(x,y) \in C^\infty,  \text{ and } \\
& h_{jk}(x,y) = h_{jk}(0,y)+ x \wt h_{jk}(x,y), \;\ \wt h_{jk}(x,y) \in C^\infty.
\end{split}
\eeq
and so the corresponding quadratic forms for $H$ on $X_\infty$ with Dirichlet or Neumann boundary conditions defined in \eqref{qform1} may be written as
\beq\label{quadF1}
\begin{split}
 Q^\bullet(X_\infty)(\vphi,\vphi)= & \int_{0}^{x_0} \int_{\p X} \biggl( \bigl|x \p_x \vphi\bigr|^2 + \sum_{j,k=1}^n h^{jk}(x,y) (x \p_{y_j} \vphi) (x \p_{y_k} \bar \vphi) + (V-\nsq)|\vphi|^2 \biggr)\frac{\sqrt{|h(x)|}}{x^{n+1}} dy dx,  \\
& \vphi \in \mcw_0^1(X_\infty), \text{ if } \bullet=D, \;\  \vphi  \in \bar \mch_0^1(X_\infty), \text{ if } \bullet=N, \text{ and } (h^{jk})= (h_{jk})^{-1}.
\end{split}
\eeq

We want to show that the quadratic forms
 $Q^\bullet(X_\infty),$ $\bullet=D,N,$ can be bounded from above and below by quadratic forms associated with the product metric where $h(x)$ is replaced by $h(0)$ and  $V(x,y)$ is replaced by  potentials which depend on $x$ only.    If $|h(x)|=|\det h_{jk}(x,y)|$  is the volume element of the metric $h(x),$ and $\bigl(h^{jk}(x,y)\bigr)=\bigl(h_{jk}(x,y)\bigr)^{-1},$  is the inverse of the matrix $\bigl(h_{jk}(x,y)\bigr),$   we deduce from \eqref{coordh} that there exists $\ga>0$ and $x_0$ small such that for $ x\in(0,x_0),$  $(1-\ga x)>\ha$   and for all $\xi \in \mr^n,$ 
 \beq\label{ineqh}
 \begin{split}
  (1-\ga x)^\ha &\sum_{j,k=1}^n h^{jk}(0,y)  \ \xi_j \xi_k \leq \sum_{j,k=1}^n h^{jk}(x,y) \xi_j \xi_k \leq 
  (1+\ga x)^\ha \sum_{j,k=1}^n h^{jk}(0,y) \xi_j \xi_k,  \text{ and } \\
&  (1-\ga x)^\ha \sqrt{|h(0)|} \leq \sqrt{|h(x)|} \leq (1+\ga x)^\ha \sqrt{|h(0)|}.
 \end{split}
\eeq

Recall that Theorems \ref{main2}, \ref{main1} and \ref{doubleT} require that, if $x=e^{-\rho},$ the potential $V(x,y)$ satisfies
\[
  -V_0(\rho)- aV_1(\rho) \leq V(x,y) \leq  -V_0(\rho)+ aV_1(\rho), \text{ where } a>0 \text{ is a constant }
\]
and $V_0(\rho)$ and $V_1(\rho)$ satisfy, for $\rho \in (\rho_0,\infty),$  one of the following three assumptions:

\beq\label{ASV}
\begin{split}
& V_0(\rho)= c \rho^{-2+\del} \text{ and } V_1(\rho)= \rho^{-2+\del} (\log \rho)^{-\eps}, \;\ \eps>0, \text{ in Theorem \ref{main2} } \\
& V_0(\rho)= c \rho^{-2}  \text{ and } V_1(\rho)= \rho^{-2} (\log \rho)^{-\eps}, \;\ \eps>0, \text{ in Theorem \ref{main1} } \\
&V_0(\rho)= \oq \rho^{-2} -  \oq \sum_{j=1}^N \mcg_j(\rho)  + c_N \mcg_N(\rho)  \text{ and } V_1(\rho)= \mcg_N(\rho) (\log\rho)^{-\eps}, \;\ \eps>0, \\
& \text{ with }  \mcg_N(\rho) \text{  defined in \eqref{def-GN},} \text{ in Theorem  \ref{doubleT}.} 
\end{split}
\eeq

We will prove the following
\begin{prop}\label{abobel}  Let $Q^\bullet(X_\infty),$ $\bullet=N,D,$ is given by \eqref{quadF1}, then for $x_0$ small enough  and $\ga$ given by  \eqref{ineqh},
\beq\label{quadF2}
\mcq_-^\bullet(X_\infty)(\vphi,\vphi) \leq Q^\bullet(X_\infty)(\vphi,\vphi) \leq \mcq_+^\bullet(X_\infty)(\vphi,\vphi),
\eeq
where $\mcq_\pm^\bullet(X_\infty)$  are the quadratic forms defined by
\beq\label{quadF3}
\begin{split}
&  \mcq_{\pm}^\bullet(X_\infty)(\vphi,\vphi)= \\
  \int_{0}^{x_0} \int_{\p X}  (1\pm \ga x) \biggl( & \bigl|x \p_x \vphi\bigr|^2  + \sum_{j,k=1}^n h^{jk}(0,y) (x \p_{y_j} \vphi) (x 
  \p_{y_k} \bar \vphi)\biggr) \frac{\sqrt{|h(0)|}}{x^{n+1}} dy dx + \\
    \int_{0}^{x_0} \int_{\p X}  \biggl(-V_0(-\log x)\pm & \ a V_1(-\log x)+ x W^\pm(x)-\nsq \biggr)|\vphi|^2 (1\pm \ga x)^{n+1} \frac{\sqrt{|h(0)|}}{x^{n+1}} dy dx,   \\
 \vphi  \in \mcw_0^1(X_\infty),  \text{ if } \bullet=D,  & \;\  \vphi \in \bar \mch_0^1(X_\infty), \text{ if } \bullet=N, \text{ and } (h^{jk}(0))= (h_{jk}(0))^{-1},
\end{split}
\eeq
where 
\beq\label{Wpm}
\begin{split}
&  W^\pm (x)=  \bigl(-V_0(-\log x) \pm a V_1(-\log x)-\nsq\bigr)  F_{\pm}(x), \;\  F_\pm(x)= \frac{1}{x}\biggl(\frac{(1\pm \ga x)^\ha}{(1\pm \ga x)^{n+1}}-1\biggr).
\end{split}
\eeq
Moreover, $\mcq_\pm^\bullet(X_\infty)$ are respectively associated with the operators
\beq\label{GPM}
\begin{split}
 & \mcm_\pm = \Delta_{g_+} -V_0(-\log x)  \ \pm a V_1(-\log x)+x W^\pm(x) -\nsq, \\
& \text{ where }  g_{\pm }= (1\pm \ga x)^{\frac{2}{n-1}} \bigl(  \frac{dx^2}{x^2} + \frac{h(0)}{x^2}\bigr)
\end{split}
\eeq
with $\bullet=D,N,$ for  Dirichlet or Neumann boundary conditions.
\end{prop}
\begin{proof}  We observe that because of \eqref{ineqh} 
\[
\begin{split}
& (1-\ga x) \biggl( |x \p_x \vphi|^2  +   h^{jk}(0,y) \ (x \p_{y_j} \vphi) (x \p_{y_k} \bar \vphi)\biggr) \sqrt{|h(0)|} \leq \\
&  \biggl(|x \p_x \vphi|^2  + h^{jk}(x,y) (x \p_{y_j} \vphi) (x \p_{y_k} \bar \vphi)\biggr) \sqrt{|h(x)|} \leq  \\
&  (1+\ga x) \biggl( \ |x \p_x \vphi|^2  + h^{jk}(0,y) (x \p_{y_j} \vphi) (x \p_{y_k} \bar \vphi)\biggr) \sqrt{|h(0)|},
\end{split}
\]
and 
\[
\bigl(V(x,y)-\nsq\bigr)(1-\ga x)^\ha  \sqrt{|h(0)|}\leq  \bigl(V(x,y)-\nsq\bigr) \sqrt{|h(x)|} \leq   \bigl(V(x,y)-\nsq\bigr)(1+\ga x)^\ha  \sqrt{|h(0)|}.
\]
But we may write
\[
\begin{split}
& (1\pm \ga x)^{\ha}=  (1\pm \ga x)^{n+1} (1+ x F_{\pm}(x)),\text{ where }  F_{\pm}(x)= \frac{1}{x}\biggl(\frac{(1\pm \ga x)^{\ha}}{1(\pm \ga x)^{n+1}}-1\biggr),
\end{split}
\]   
and so
\[
\begin{split}
& \bigl(V(x,y)-\nsq\bigr)(1\pm \ga x)^\ha = V(x,y)(1\pm \ha x)^{n+1}(1+x F_\pm(x)).
\end{split}
\]
Therefore
\[
\begin{split}
 \bigl( -V_0(\log x)-  a & V_1(-\log x) ) (1- x F_-(x))(1+\ga x)^{n+1} \sqrt{|h(0)|} \leq \\
&  \bigl(V(x,y)-\nsq\bigr) \sqrt{|h(x)|}  \leq \\
 \bigl( -V_0(\log x)+  a & V_1(-\log x) ) (1+ x F_+(x))(1+\ga x)^{n+1} \sqrt{|h(0)|} \text{ and } \\
\end{split}
\]
This proves \eqref{quadF2}. 
Notice that for $g_{\pm}$ as in \eqref{GPM},
\[
\begin{split}
& \sqrt{|g_{\pm}|} = \frac{(1\pm \ga x)^{\frac{n+1}{n-1}}}{x^{n+1}} |\sqrt{|h(0)}|,  \text{ and } \\
\Delta_{g_\pm }= & \ -\frac {x^{n+1}} {(1\pm \ga x)^{n+1}\sqrt{|h(0)|}}  \ \p_x\bigl( \frac{(1\pm \ga x)}{x^{n+1}}  x^2\sqrt{|h(0)|} \p_x\bigr)
 - \\
 & \frac {x^{n+1}}{(1\pm \ga x)^{n+1}\sqrt{|h(0)|}} \p_{y_j}\bigl( \frac{(1\pm \ga x)}{x^{n+1}} \sqrt{|h(0)|}  x^2 h^{jk}(0)\p_{y_k}\bigr),
 \end{split}
\]
and so, the quadratic forms associated with $g_{\pm}$  define the operators $\mcm_\pm^\bullet,$ as claimed. The ends the proof of the Proposition.

\end{proof}

We have shown that the quadratic forms $\mcq_\pm^\bullet(X_\infty)$ and $Q^\bullet(X_\infty)$  have the same domain $\mcw_0^1(X_\infty),$ when $\bullet=D$ and $ \bar \mch_0^1(X_\infty),$ if  $\bullet=N$ and moreover

\beq\label{quad-est0}
\begin{split}
&  \mcq_-^D(X_\infty)(\vphi, \vphi) \leq  Q^D(X_\infty)(\vphi,\vphi) \leq \mcq_+^D(X_\infty)(\vphi, \vphi), \;\ \vphi  \in \mcw_0^1(X_\infty),  \\
&  \mcq^N(X_\infty)(\vphi, \vphi) \leq  Q^N(X_\infty)(\vphi,\vphi) \leq \mcq^N(X_\infty)(\vphi, \vphi), \;\ \vphi  \in \bar \mch_0^1(X_\infty). 
\end{split}
\eeq

Notice that the $L^2(X_\infty)$ spaces defined with respect to $g$ or $g_\pm$ are the same, but are equipped with different  but equivalent norms, and there are constants $C_j^\pm,$ $j=1,2$ such that
\beq\label{squeeze}
\begin{split}
& C^-_1 || \vphi ||_{L^2_{{}_{g_-}}} \leq  || \vphi ||_{L^2_{{}_{g}}} \leq C^-_2 || \vphi ||_{L^2_{{}_{g_-}}},  \\
& C^+_1 || \vphi ||_{L^2_{{}_{g_+}}} \leq  || \vphi ||_{L^2_{{}_{g}}} \leq C^+_2 || \vphi ||_{L^2_{{}_{g_+}}}.
\end{split}
\eeq
  If we put together \eqref{quad-est0} and \eqref{squeeze}  we obtain 
\[
\begin{split}
& \frac{1}{C_2^-} \frac{\mcq_-^D(X_\infty)(\vphi,\vphi)}{\lan \vphi,\vphi\ran_{{}_{L_{g_-}^2}}} \leq 
\frac{Q^D(X_\infty)(\vphi,\vphi)}{\lan \vphi,\vphi\ran_{{}_{L_{g}^2}}} \leq  \frac{1}{C_1^+} \frac{\mcq_+^D(X\infty)(\vphi,\vphi)}{\lan \vphi,\vphi\ran_{{}_{L_{g_+}^2}}}, \;\ \vphi  \in \mcw_0^1(X_\infty), \\
& \frac{1}{C_2^-} \frac{\mcq_-^N(X_\infty)(\vphi,\vphi)}{\lan \vphi,\vphi\ran_{{}_{L_{g_-}^2}}} \leq 
 \frac{Q^N(X_\infty)(\vphi,\vphi)}{\lan \vphi,\vphi\ran_{{}_{L_{g}^2}}} \leq  \frac{1}{C_1^+} \frac{\mcq_+^N(X_\infty)(\vphi,\vphi)}{\lan \vphi,\vphi\ran_{{}_{L_{g_+}^2}}},  \;\ \vphi  \in \bar \mch_0^1(X_\infty).\\
\end{split}
\]

We will use $\mcm_\pm^\bullet$ to indicate the operator $\mcm_\pm$ with $\bullet=D,N.$ We also remark that one may extend the metrics $g_\pm$ to the manifold $X,$ so that it becomes an AHM, and as a consequence of Theorem \ref{specDN}  we obtain
\beq\label{specmod}
\sigma_{ess}(\mcm_\pm^\bullet)= [0,\infty), \;\ \bullet=D,N.
\eeq

We now appeal to the following characterization of the eigenvalues of a self-adjoint operator, see for example 
 page 1543 of \cite{Dun} or Theorem 3 of \cite{Sten}:
\begin{theorem}\label{min-max}  Let $H$ be a separable Hilbert space with inner product $\lan u,v\ran$ and let $A$ be a self adjoint operator corresponding to a semi-bounded quadratic form $Q$ with domain $\mcd(Q).$  Suppose that the essential spectrum of $A,$  $\sigma_{ess}(A)=[0,\infty),$ and its point spectrum satisfies
\[
\sigma_{pp}(A)=\{ \la_1 \leq \la_2 \leq \ldots \}.
\]
 For $u\in \mcd(Q),$ $u\not=0,$ let $R(u)=\frac{Q(u,u)}{\lan u,u\ran}$ denote the Rayleigh quotient, and  for $n\in \mn,$ let
\[
\mu_n= \inf \biggl\{ \max \bigl\{ R(u),  u \in \mcb \text{ such that } \mcb \subset \mcd(Q) \text{ is a subspace } 
\dim \mcb=n\bigr\}\biggr\}.
\]
Then $\mu_n\leq 0.$ If $\mu_n=0$ then $A$ has at most $n-1$ eigenvalues $\la_j<0,$ counted with multiplicity.  If $\mu_n<0,$ then $\mu_n=\la_n,$ is the n-th eigenvalue of $A$ counted with multiplicity.
\end{theorem}

The important aspect of this characterization is that there is no orthogonality required, as would be the case if we switched the order of $\max$ and $\min.$ As a consequence of Theorem \ref{specDN} and Theorem \ref{min-max} and \eqref{specmod} we arrive at the following:
\begin{prop}\label{eigen-ope} Let $\la_j(\Upsilon_\infty^D)$ and $\la_j(\mcm_\pm^D)$ denote the eigenvalues of the operators  $\Upsilon_\infty$ and $\mcm^\pm,$ with Dirichlet boundary conditions.  Similarly, let $\mu_j(\Ups_\infty^N)$ and $\mu_j(\mcm_\pm^N)$ denote the eigenvalues of these operators with Neumann boundary conditions.  If $\mcm_-^\bullet$  has finitely many eigenvalues, so do  $\mcm_+^\bullet$ and $\Ups_\infty^\bullet,$ $\bullet=D,N.$ If $\mcm_+^\bullet$ has infinitely many eigenvalues, so do $\mcm_-^\bullet$ and $\Ups_\infty^\bullet,$ $\bullet=D,N$ and  if  $C_j^\pm,$ $j=1,2,$ are as defined in \eqref{squeeze}, then for all $j,$
\beq\label{eigen3}
\begin{split}
& \frac{1}{C_2^-} \la_j(\mcm_-^D) \leq \la_j(\Ups_\infty^D) \leq \frac{1}{C_1^+} \la_j(\mcm_+^D), \\
& \frac{1}{C_2^-} \mu_j(\mcm_-^N) \leq \mu_j(\Ups_\infty^N) \leq \frac{1}{C_1^+} \mu_j(\mcm_+^N).
\end{split}
\eeq

In particular, in the case there exist infinitely many eigenvalues, then  for all  $E<0,$
\beq\label{count-rel}
\begin{split}
& N_{C_2^- E}(\mcm_-^\bullet) \leq N_E(\Ups_\infty^\bullet) \leq N_{C_1^+ E}(\mcm_+^\bullet), \;\ \bullet=D,N,
\end{split}
\eeq
where $N_E(T)$ is the counting function of negative eigenvalues defined in \eqref{counting}.
\end{prop}

\section{ The Asymptotic Behavior of $N_E(\mcm_\pm^\bullet),$ $\bullet=D,N$  as $E\rightarrow 0$}

We will show that  under the hypotheses of either one of  Theorems \ref{main2}, \ref{main1} or \ref{doubleT}, and  if  $\rho_0=-\log x_0$ is  large enough, we have two possibilities: either both $\mcm_\pm^D$ and $\mcm_\pm^N$  have no negative eigenvalues, or  both have infinitely many.  In the latter case,  the iterated logarithms of the counting functions $N_E(\mcm_\pm^\bullet),$ $\bullet=D,N,$ defined in \eqref{counting} have the same asymptotic behavior $E\rightarrow 0,$ according to the asymptotic behavior of the potential as in  Theorems \ref{main2}, \ref{main1} and \ref{doubleT}.  More precisely, we will prove the following
\begin{prop}\label{asym-M}  Let $\mcm_\pm^\bullet,$ $\bullet=D,N,$ be defined as above and let $N_E(\mcm_\pm^\bullet)$ denote the corresponding counting function of eigenvalues. We have the following:
\begin{enumerate}[T.1]
\item \label{itT1} Suppose $V_0(\rho)$ and $V_1(\rho)$ satisfy \eqref{hyp-pot}.  If $\rho_0$ is large enough and $\del<0,$ then $\mcm_\pm^\bullet$ has no negative eigenvalues, but if $\del\in (0,2),$ then $N_E(\mcm_\pm^\bullet)$ satisfies 
\eqref{subc}.
\item \label{itT2} Suppose $V_0(\rho)$ and $V_1(\rho)$ satisfy \eqref{hyp-potTC}.  If $\rho_0$ is large enough and $c<\oq,$ then  $\mcm_\pm^\bullet$ has no negative eigenvalues, but if  $c>\oq,$ then $N_E(\mcm_\pm^\bullet)$ satisfies \eqref{bdhs1}.
\item \label{itT3} Suppose $V_0(\rho)$ and $V_1(\rho)$ satisfy  \eqref{def-GN}.  If $\rho_0$ is large enough and $c_N<\oq,$ then  $\mcm_\pm^\bullet$ has no negative eigenvalues, but if $c_N>\oq$  then $N_E(\mcm_\pm^\bullet)$ satisfies \eqref{bdhs4}.
\end{enumerate}
\end{prop}

These results, together with equations \eqref{count-rel}, \eqref{eigen-est1} and \eqref{spec-X0} respectively prove Theorems \ref{main2}, \ref{main1} and \ref{doubleT}.

We  will consider the Dirichlet and Neumann eigenvalue problems in $X_\infty=\{x\leq x_0\}$ for the operators $\mcm_\pm^\bullet,$ $\bullet=D,N,$ defined above.  We will drop the $\pm$ sub-indices and work with $\ga, a \in \mr.$  We will assume that $x_0$ is small enough so that
$x|\ga|<\ha$ for $x\in (0, x_0).$  We will work with the metric 
\[
\mcg= (1+\ga x)^{\frac{2}{n-1}}\bigl(\frac{dx^2}{x^2}+ \frac{h(0)}{x^2}\bigr).
\]
We find that
\[
\Delta_\mcg= -x^{n+1} f^{-(n+1)} \p_x\left( f^{n-1} x^{1-n} \p_x\right) + x^2 f^{-2} \Delta_{h(0)}, \text{ where } \;\ f(x)=(1+\ga x)^{\frac{1}{n-1}}.
\]
To get rid of the factor $\nsq$ in $\mcm_\pm,$ we observe that
\[
\begin{split}
& x^{-\novt}(\Delta_\mcg-\nsq) x^{\novt}= - f^{-(n+1)} x\p_x(f^{n-1} x\p_x) + x^2 f^{-1} \Delta_{h(0)}-  x A(x), \text{ where }\\
 & A(x)=-\novt f^{-(n+1)}(x)\p_x( f^{n-1}(x)) + \frac{n^2}{4x}( f^2(x)-1).
 \end{split}
\]
So we want to study the eigenvalue problems corresponding to the operators $x^{-\novt}\mcm_\pm x^{\novt},$ which are of the form
\[
\begin{split}
 \bigl(- f^{-(n+1)} (f^{n-1} x\p_x)^2 + & x^2 f^{-2} \Delta_{h(0)}+ U(x)+ x \mcw(x) +E\bigr)u^\bullet = 0, \;\ E>0, \bullet=D,N, \\
&  u^D(x_0,y)= 0 \text{ or } \p_x u^N(x_0,y)=0, \\
 \text{ where }  \ U(x)= V_0(-\log x) & \ + a V_1(-\log x) \text{ and }  \mcw(x)= (U(x)-\nsq) F(x) - A(x), \;\ F\in C^\infty.
\end{split}
\]
We multiply the equation by $f^{n+1},$ and use that for $x$ small $f^{n+1}(x)=1+ x \wt f(x),$ $\wt f\in C^\infty,$  and we arrive at
\beq\label{eq0112}
\begin{split}
 \bigl(-(f^{n-1} x\p_x)^2 + & \ x^2 f^{n-1} \Delta_{h(0)} + U(x)+ x \wt \mcw(x) +E\bigr)u^\bullet = 0, \;\ E>0, \bullet=D.N\\
& u^D(x_0,y)= 0 \text{ or } \p_x u^N(x_0,y)=0, \\
& \wt \mcw= \mcw + \wt f(U+x\mcw+ E), \;\ \wt f \in C^\infty.
\end{split}
\eeq
Noice that $\wt \mcw(x)$ has a term of the form $E x \wt f(x),$ $f\in C^\infty.$ So one should keep in mind that  it depends on $E,$ but it will not affect the estimates below because this term goes to zero if $E\rightarrow 0.$.

Since $f^{n-1}= 1+\ga x,$ we define $r$ to be such that
\[
\frac{d r}{r}= \frac{dx }{x (1+\ga x)}, \;\ r=0 \text{ if } x=0,
\]
and therefore
\beq\label{xtil}
r= \frac{x}{1+\ga x} \text{ and so }  x= \frac{r}{1-\ga r}= r+ r^2 X(r), \;\ X\in C^\infty([0,r_0]), \;\ r_0 \text{ small enough.}
\eeq

Therefore, after the change of variables, equation \eqref{eq0112} becomes
\beq\label{eq0113}
\begin{split}
 \bigl(-(r\p_r)^2 + & \ r^2 (1+ rF(r)) \Delta_{h(0)} + U(x(r))+ x(r) \wt \mcw(x(r)) +E\bigr)u^\bullet = 0,\\
&  u^D(r_0,y)= 0 \text{ or } \p_r u^N(r_0,y)=0, \\
& \text{ where }  \;\ E>0,  \ \bullet=D,N, \text{ and } F \in C^\infty([0,r_0]).  
\end{split}
\eeq

We will need the following fact:
\begin{lemma} \label{chvar}  Suppose that $x=x(r)= r+ r^2X(r),$ with $X(r)\in C^\infty([0,r_0])$ and  $r_0$ is small enough. Let  $\rho=-\log x$ and let $\mcg_{(j)}(\rho)$ be as defined in \eqref{def-GN}, then
\beq\label{logxr}
\begin{split}
& \log_{(j)}(x(r)^{-1})= \log_{(j)}(r^{-1}) \biggl(1+ \frac{r}{(\log r) (\log_{(j)}r^{-1})} \Tau_j(r)\biggr),  \;\ \Tau_j(r)\in C^\infty((0, r_0]) \cap C^0([0, r_0]), \\
& \text{ and for } \alpha \in \mr, \\
& (\log_{(j)} (x(r))^{-1})^{\alpha}= (\log_{(2)} r^{-1})^{\alpha}\biggl( 1+ \frac{r}{(\log r^{-1}) (\log_{(j)} r^{-1})} T_{j,\alpha}(r)\biggr), \;\ T_{j,\alpha}\in C^\infty((0, r_0]) \cap C^0([0, r_0]).
\end{split}
\eeq
As a consequence of these we find that if $\mcg_{(j)}$ is given by \eqref{def-GN}, then
\beq\label{GNR}
\begin{split}
 & \mcg_{(j)}(\log( x(r)^{-1} )=  \ \mcg_{(j)}(\log (r^{-1}))\bigl(1+ \frac{r}{(\log r) (\log_{(2)} r^{-1})} \Tau_j(r)\bigr), \\ 
 & \text{ where }  \Tau_j \in C^\infty((0, r_0]) \cap L^\infty([0, r_0]).
 \end{split}
\eeq
 If  $V_0(\rho)$ and $V_1(\rho),$ $\rho=-\log x,$ are given by one of the alternatives of \eqref{ASV},  and $\vrho=-\log r,$ then
\beq\label{GNR1}
\begin{split}
& V_0(\rho)= V_0(\vrho) + e^{-\vrho} \mcv_0(\rho), \;\ \mcv_0\in C^\infty([\vrho_0,\infty)) \cap L^\infty([\vrho_0, \infty)), \\
& V_1(\rho)= V_1(\vrho) + e^{-\vrho} \mcv_1(\vrho), \;\ \mcv_1\in C^\infty([\vrho_0,\infty) \cap L^\infty([\vrho_0,\infty)).
\end{split}
\eeq
\end{lemma}
\begin{proof}  The point is that since for $t$ small, $\log(1+t)= t f(t),$ $f\in C^\infty,$ we find that
\[
\begin{split}
& \log x(r)= \log\bigl( r(1+ rX(r))\bigr)= \log r+ \log (1+ rX(r)) = \log r( 1+ \frac{r}{\log r} T(r)), \\
&  T(r)= X(r) f\bigl(r X(r))\in C^\infty ([0, r_0]).
\end{split}
\]

Similarly if $\alpha \in \mr,$ we have that $(1+t)^\alpha= 1+ t f_\alpha(t),$ $f_\alpha\in C^\infty,$  and so
\[
\begin{split}
& (-\log x(r))^\alpha = (-\log r)^{\alpha} ( 1+ \frac{r}{\log r} T(r))^\alpha= (-\log r)^{\alpha} (1+  \frac{r}{\log r}  T_{\alpha}(r)), \\ 
& T_{\alpha}= T(r) f_{\alpha}\bigl( \frac{r T(r)}{\log r}\bigr) \in C^\infty((0, r_0]) \cap L^\infty([0, r_0]).
\end{split}
\]
Using the same ideas we obtain
\[
\begin{split}
& \log_{(2)}\bigl( x(r)^{-1}\bigr)= \log \biggl( \log (r^{-1}) \bigl(1+\frac{r}{\log r} T(r)\bigr) \biggr)= \log_{(2)}(r^{-1}) + \log \bigl(1+\frac{r}{\log r} T(r)\bigr)= \\
&(\log_{(2)}( r^{-1}))\biggl( 1+ \frac{r}{(\log r)(\log_{(2)} r^{-1})} T_1(r)\biggr), \;\ T_1\in C^\infty((0, r_0]) \cap L^\infty([0, r_0]).
\end{split}
\]
Using induction we find  that 
\[
\log_{(j)}( x(r)^{-1})= (\log_{(j)} (r^{-1}))\bigl( 1+ \frac{r}{(\log r)(\log_{(j)} r^{-1})} T_j(r)\bigr), \;\ T_j \in C^\infty((0, r_0]) \cap L^\infty([0, r_0]),
\]
and this proves the first equation in \eqref{logxr}.  The second equation in \eqref{logxr} and the equations \eqref{GNR}  and \eqref{GNR1} follow directly from \eqref{logxr}.
\end{proof}

If we replace $r= e^{-\vrho},$ $\vrho\in [\vrho_0,\infty)$ in \eqref{eq0113}   (we have used $\rho=-\log x$ before, and  here we are using $\vrho=-\log r,$ and these are not the same) and to simplify the notation we will use $ \mcx(\vrho)=\wt \mcw(x(e^{-\vrho})),$ and  $ B(\vrho)= F(e^{-\vrho}),$ so   equation \eqref{eq0112} becomes
\beq\label{BVP}
\begin{split}
& \mcm u^\bullet= -E u^\bullet, \;\ \bullet=D,N \;\  u^D(\vrho_0,y)= 0 \text{ or } \p_\vrho u^N(\vrho_0,y)=0, \text{ where } \\
& \mcm= -\p_\vrho^2+  \ q(\vrho) e^{-2\vrho} \Delta_{h(0)} + U(e^{-\vrho})  +e^{-\vrho} \mcx(\vrho), \\
 &  q(\vrho)= 1+ e^{-\vrho} B(\vrho), \;\  B\text{ and } \mcx \in L^\infty([\vrho_0,\infty)) \cap C^\infty([\vrho_0,\infty)).
\end{split}
\eeq

We decompose $u^\bullet(\vrho,y)$ in Fourier series with respect to the eigenfunctions of the Laplacian $\Delta_{h(0)}$ on $\p X:$
\beq\label{eigdec}
\begin{split}
u^\bullet(\vrho,y)= & \sum_{j=0}^\infty  u_{j}^\bullet(\vrho)  \psi_j(y),  \;\ u_{j}^\bullet(\vrho)= \lan u^\bullet (\vrho,y), \psi_j(y)\ran_{{}_{L^2(\p X,h(0))}}, \text{ where } \\
& \Delta_{h(0)} \psi_j = \zeta_j \psi_j,  \;\ 0= \zeta_0 <\zeta_1 <  \zeta_2 \leq \zeta_3 \ldots 
\end{split}
\eeq

Let 
\beq\label{mult}
\begin{split}
& \mcy_j= \text{ the eigenspace  corresponding to } \zeta_j  \text{ and define }  \\
& m(\zeta_j)= \operatorname{dim} \mcy_j= \text{ multiplicity of } \zeta_j,
\end{split}
\eeq
and so we have that
\beq\label{mj}
\begin{split}
& L^2(X_\infty)= \bigoplus_{j=1}^\infty \mcy_j, \text{ and } \mcm^\bullet= \bigoplus_{j=1}^\infty \mcm_j^\bullet, \bullet=D,N, \text{ where } \\
& \mcm_{j}= -\left( \frac{d}{d\vrho}\right)^2 + e^{-2\vrho}q(\vrho)\zeta_j + U(e^{-\vrho}) +  e^{-\vrho} \mcx(\vrho).
\end{split}
\eeq
For each $j,$  $\mcm_{j}^\bullet,$ $\bullet=D,N,$  are self-adjoint operators and  $\sigma_{ess}(\mcm_j^\bullet)=[0,\infty),$ see for example \cite{Glad}.  We  prove in Appendix \ref{SPEC}, for the convenience of the reader,  that they have no eigenvalues in $[0,\infty),$ so  they have only negative eigenvalues which only possibly accumulate at zero.  It also follows from \eqref{mj} that if $E>0,$
\[
\operatorname{dim}P_{(-\infty,-E)}(\mcm^\bullet)= \sum_{j=1}^\infty\ m(\zeta_j) \operatorname{dim}P_{(-\infty,- E)}(\mcm_j^\bullet), \;\ \bullet=D,N,
\]
or in other words, 
\beq\label{decompM}
\begin{split}
& N_E(\mcm^\bullet)= \sum_{j=1}^\infty m(\zeta_j)N_E(\mcm_j^\bullet), \;\ \bullet= D,N,   
\end{split}
\eeq
  where as above, $N_E(T) $  denotes the counting function \eqref{counting}.

The eigenfunctions $\phi_{j}^\bullet(\vrho,E),$ $\bullet=D,N,$  with eigenvalue $-E$ satisfy
\beq\label{PDEk0}
\begin{split}
& (\mcm_{j}^\bullet +E)\phi_{j}^\bullet= 0,\\
& \phi_{j}^D(\vrho_0,E)=0, \;\  \p_\vrho \phi_{j}^D(\vrho_0,E)=1, \text{ or } \\
 & \phi_{j}^N(\vrho_0,E)=1, \;\  \p_\vrho \phi_{j}^N(\vrho_0,E)=0,
\end{split}
\eeq
just notice that  by dividing the equation by a constant, we can always assume either the function or its derivative is equal to one at $\vrho_0.$  But for  $E > 0,$  we will consider the Cauchy problems
\beq\label{EIGj}
\begin{split}
( \mcm_{j}^\bullet+E) u_{j}^\bullet(\vrho,E)=& \ 0,  \;\ \bullet=D, N  \text{ on } (\vrho_0,\infty), \text{ with boundary conditions } \\
& u_{j}^D(\vrho_0,E)=0, \;\  \p_\vrho u_{j}^D(\vrho_0,E)=1   \text{ or }  \\
&  u_{j}^N(\vrho_0,E)=1, \;\   \p_\vrho u_{j}^N(\vrho_0,E)=0,
\end{split}
\eeq
which have  unique solutions in $X_\infty.$  Even though  $u_j^\bullet(\vrho,E)$ exist for every $E,$ for $\vrho\in (\vrho_0,\infty),$ and are not necessarily eigenfunctions because they may not be in $L^2.$  In fact, it follows from Theorem \ref{olver} in Appendix \ref{SPEC} that   a solution $u_j^\bullet(\vrho,E)$  of \eqref{EIGj} is an eigenfunction if and only if $u_j^\bullet(\vrho,E)\sim C e^{-\vrho \sqrt{E}}$ as $\vrho\rightarrow \infty.$ The key point of the proof Proposition \ref{asym-M} is
\begin{prop}\label{lmzero}  Let $u_{j}^\bullet(\vrho,E),$ $\bullet=D,N$ be the unique solutions of the corresponding Cauchy problems in \eqref{EIGj}.   Let $Z_{j}^\bullet(E)$ denote the number of its zeros, which are  different from $\vrho=\vrho_0$ in the case $\bullet=D.$  If $E<0,$ then
\beq\label{zer-eigen}
N_E(\mcm_j^\bullet)= Z_j^\bullet(E), \;\ \bullet=D, N,
\eeq
and as a consequence of  \eqref{decompM} we have
\beq\label{NEV}
N_E(\mcm^\bullet)= \sum_{j=1}^\infty m(\zeta_{j})  Z_{j}^\bullet(E), \;\ \bullet=D,N.
\eeq
\end{prop}

This  result is somewhat well known and its proof is essentially, but not quite, the same as the proof of Theorem XIII.8 of \cite{RS4}. For the convenience of the reader, we  provide the details in Appendix \ref{Zeros}.

\subsection{ The Set up of the Problems}\label{asymp}   Now we have to count the zeros of solutions of \eqref{EIGj}. We set up the general type of problem for $U(e^{-\vrho})=-V_0(\vrho)+ a V_1(\vrho)$ with $V_j(\vrho),$ $j=0,1,$ satisfying one of the alternatives in \eqref{ASV}.  The arguments we use do not depend on the boundary condition, so we work with the Dirichlet problem in \eqref{EIGj}. We  consider the problem

\beq\label{pde-Rd}
\begin{split}
& \left(-\frac{d^2}{d\vrho^2}- V_0(\vrho) + e^{-2\vrho}(1+ B(\vrho)e^{-\vrho}) \mu +E  +  aV_1(\vrho) + e^{-\vrho} \mcx(\vrho) \right) u=0, \;\ E>0, \\
& \quad \quad \quad \quad \quad \quad \quad \quad \quad u(\vrho_0)=0, \;\ \p_\vrho u(\vrho_0)=1, 
\end{split}
\eeq
 where $B(\vrho), \mcx(\vrho)\in C^\infty([\vrho_0,\infty))\cap L^\infty([\vrho_0,\infty)).$   We will denote
 \beq\label{mcr-mcp}
\begin{split}
& \mcr(\vrho)= aV_1(\vrho) + e^{-\vrho} \mcx(\vrho), \text{ and } \\
& \mcp(\vrho) =  \mu  e^{-2\vrho}(1+ e^{-\vrho}B(\vrho))+E,
\end{split}
 \eeq
 and \eqref{pde-Rd} is reduced to
 \beq\label{pde-Rd1}
 \begin{split}
& \left(-\frac{d^2}{d\vrho^2}- V_0(\vrho) + \mcp(\vrho)+ \mcr(\vrho) \right) u=0, \;\ E>0, \\
& \quad \quad \quad u(\vrho_0)=0. 
\end{split}
\eeq
 
We will deal with each one of the cases  \eqref{hyp-pot}, \eqref{hyp-potTC} and in general  \eqref{def-GN} separately.
\subsection{Proof of Item T.\ref{itT1} of Proposition \ref{asym-M}}   In this case 
\[
V_0(\vrho)=c \vrho^{-2+\del} \text{ and } V_1(\vrho)= a \vrho^{-2+\del}(\log \vrho)^{-\eps}.
\]
 We multiply  equation \eqref{pde-Rd1} by $\vrho^{2-\del},$  and notice that
\begin{gather*}
\vrho^{-\ha(1-\dovet)}(\vrho^{2-\del} \frac{d^2}{d\vrho^2} )  \vrho^{\ha(1-\dovet)} = (\vrho^{1-\dovet}\frac{d}{d\vrho})^2-\oq(1-\frac{\del^2}{4}) \vrho^{-\del}.
\end{gather*}
So if $u(\vrho)= \vrho^{\ha(1-\dovet)} w(\vrho),$ then \eqref{pde-Rd1} becomes
\beq\label{conj-pde-R1d}
\begin{split}
 \biggl( & -(\vrho^{1-\dovet} \frac{d}{d\vrho})^2  -c+ \oq(1-\frac{\del^2}{4}) \vrho^{-\del}+ \mce(\vrho) \biggr) w=0, \\
& w(\vrho_0)=0, \text{ where } \\
& \mce(\vrho)= \vrho^{2-\del}( \mcp(\vrho)+ \mcr(\vrho)),  \text{ and } \\
& \mcr(\vrho)=  a \vrho^{-2+\del} (\log \vrho)^{-\eps}+  \ e^{-\vrho} \mcx(\vrho), \;\  \eps>0, \\
& \mcp(\vrho)= \mu e^{-2\vrho}\bigl(1+ e^{-\vrho} B(\vrho)\bigr)+E.
\end{split}
\eeq

Set $t=\frac{2}{|\del|} \sqrt{c} \ \vrho^{\dovet},$ and in this case \eqref{conj-pde-R1d} becomes
\beq\label{conj-pde-del}
\begin{split}
& \left(-\frac{d^2}{dt^2} - 1 +\frac{1}{c} \mce (\vrho(t)) + \frac{1}{4c}(1-\frac{\del^2}{4})\left(\frac{ t|\del|}{2\sqrt{c}}\right)^{-2}\right) w=0, \;\  t=\sqrt{c}\ \frac{2}{|\del|} \vrho^{\dovet}, \\
& \quad \quad \quad w(t_0)=0.
\end{split}
\eeq

The case $\del<0:$ Since $\mcx(\vrho)\in L^\infty,$ it follows from the assumptions on the decay of $V_0(\vrho)$ and $V_1(\vrho)$ that
$\vrho^{2-\del}\mcr(\vrho)\rightarrow 0$ as $\vrho\rightarrow \infty.$ Since  $\mcp(\vrho)>0$ for $\vrho$ large,  there exists $t_0>0$ independent of $E$ and $\mu$ such that for $t<t_0,$ 
\[
\begin{split}
& U(t)=-1 + \frac{1}{c} \mce(\vrho(t))+ \frac{1}{4c}(1-\frac{\del^2}{4})\left(\frac{ t|\del|}{2\sqrt{c}}\right)^{-2}> \\
& -1 + \frac{1}{4c}\vrho^{2-\del}\mcr(\vrho) + \oq(1-\frac{\del^2}{4})\left(\frac{ t|\del|}{2\sqrt{c}}\right)^{-2}>0,
\end{split}
\]
and so   $w(t)$ is a solution of a differential equation 
\beq\label{DN-cond}
\begin{split}
\frac{d^2 w}{dt^2}& = U(t) w, \; t<t_0, \text{ with } U(t)>0, \\
& w(t_0)=0 
\end{split}
\eeq
We have three possibilities for $w'(t_0):$  $w'(t_0)=0,$ $w'(t_0)>0$ or $w'(t_0)<0.$ If $w'(t_0)=0,$  then by uniqueness, $w(t)=0$ for $t<t_0.$  If $w'(t_0)>0,$  since $w(t)$ is $C^\infty,$  there exists $t_1<t_0$ such that $w'(t)>0$ for 
$t_1<t\leq t_0$ and so $w(t)<0$ for $t_1<t\leq t_0$ and therefore, $w''(t)<0$ for $t_1<t\leq t_0,$ and so, $w'(t)<w'(t_0)<0$ for $t_1<t<t_0$ and $w(t)<0$ for $t_1<t<t_0.$  Repeating this argument we conclude that $w(t)<0$ for all $t<t_0.$  If $w'(t_0)<0,$ since $-w(t)$ also solves the equation, then $w(t)>0$ for all $t<t_0.$   Therefore we conclude that either $w(t)=0$ for all $t>t_0$ or $w(t)$ has no zeros for $t>t_0.$ 

 In this case  we conclude from \eqref{NEV} that for this choice of $\vrho_0,$
\beq\label{zeroE}
N_E(\mcm ^\bullet)=0, \;\ \bullet=N, D.
\eeq

The case $0<\del<2:$ In this case, in view of the discussion above,  the set of zeros of the solution $w$ is contained in the set
\[
\{\vrho>\vrho_0: -1+\frac{1}{c}  \mce_1(\vrho(t))\leq 0 \} \text{ where }  \mce_1(\vrho)= \vrho^{2-\del}\bigl(\mcp(\vrho)+\mcr(\vrho) +\oq (1-\frac{\del^2}{4}) \vrho^{-2}\bigr).
\]

We want to obtain upper and lower bounds on the number of zeros of the solution $w(t)$ of \eqref{conj-pde-del} and consequently upper bounds on the number of eigenvalues of $\mcm^\bullet,$ $\bullet=N,D.$   We begin by taking $\vrho_0$ large enough, and independently of $\mu$ and $E,$  such that
\beq\label{R01}
\left|\frac{1}{c}\vrho^{2-\del} \mcr(\vrho)+ \frac{1}{4c}(1-\frac{\del^2}{4}) \vrho^{-\del}\right| \leq \ha \text{ for all } \vrho>\vrho_0.
\eeq

Next we obtain upper bounds for the number of zeros of $w(\rho(t)),$ in $ t\in [t_0,\infty)$ where $t_0$ corresponds to $\rho_0.$  Notice that, for this choice of $\vrho_0,$ since $E$ is small,
\[
\begin{split}
\biggl\{\vrho>\vrho_0: & \ -1+\frac{1}{c}   \mce_1(\vrho(t))\leq 0 \biggr\} \subset \biggl\{\vrho\geq \vrho_0:-1\leq  \frac{1}{c} \mce_1(\vrho) \leq \tha \biggr\} \subset \\
&  \biggl\{\vrho \geq \vrho_0: \frac{\vrho^{2-\del}}{c}\bigl(E+ \mu e^{-2\vrho}(1+B(\vrho)e^{-\vrho})\bigr) \leq 2\biggr\},
 \end{split}
\]
and one may yet taken $\vrho_0$ larger if necessary,  such that $1+ e^{-\vrho}B(\rho)\geq \ha,$ and so 
\beq\label{subset-del}
\begin{split}
 \biggl\{\vrho \geq \vrho_0:  \ \frac{\vrho^{2-\del}}{c}\bigl(E+ \mu e^{-2\vrho}(1+B(\vrho)e^{-\vrho})\bigr) \leq 2\biggr\} \subset 
\biggl \{\vrho \geq \vrho_0: \frac{\vrho^{2-\del}}{c}(E+ \ha\mu e^{-2\vrho}) \leq 2 \biggr\}.
\end{split}
\eeq
 Also, if $\vrho_0$ is large,  
 \[
 \bigg\{\vrho \geq \vrho_0: \frac{\vrho^{2-\del}}{c}(E+ \ha \mu e^{-2\vrho}) \leq 2\biggr\}=[ \vrho_L,\vrho_U],
 \]
 and this is because if $F(\vrho)= \vrho^{2-\del}(E+\ha \mu e^{-2\vrho}),$ then
 \[ 
F'(\vrho)+ F''(\vrho)=(2-\del)\vrho^{-\del}(E+ \ha \mu e^{-2\vrho})(\vrho+1-\del)+ 4\mu \vrho^{1-\del} e^{-2\vrho}(\vrho-2(2-\del))>0,
\]
 and therefore, if $F'(\vrho)=0,$ then $F''(\vrho)>0,$ so if $F(\vrho)$ has a critical point, it is a local minimum.
  We then observe that
 \beq\label{subset-del1}
 \begin{split}
  & \{\vrho \geq \vrho_0: \frac{\vrho^{2-\del}}{c}(E+ \ha \mu e^{-2\vrho}) \leq 2\} =[\vrho_L, \vrho_ U]\subset 
 \bigl\{ \vrho\geq \vrho_0: \frac{E}{c}\vrho^{2-\del}  \leq 2\bigr\} \cap \bigl\{ \vrho\geq \vrho_0:  \frac{\mu}{2c} \vrho^{2-\del} e^{-2\vrho} \leq 2\bigr\}.
\end{split}
\eeq

 So if we take $\vrho_{u_\del}$ such that
 \[
 \frac{E}{c}\vrho_{u_\del}^{2-\del} = 2, \text{ so }  \vrho_{u_\del}=\left(\frac{2c}{E}\right)^{\frac{1}{2-\del}},
 \]
 then $\vrho_U \leq \vrho_{u_\del}$ and we  will get an upper bound on the number of zeros of $w(\vrho(t))$ in the interval 
 $[t_0, t_{u_\del}],$ where $\vrho_0= \vrho(t_0)$ and $\vrho_{u_\del}= \vrho(t_{u_\del}).$  This gives an upper bound on the zeros of $w(t)$ in $[t_0,\infty).$
 
Since $\vrho^{2-\del}$ is an increasing function, then 
 \[
 \frac{E}{c}\vrho^{2-\del} \leq 2 \text{ for } \vrho\leq \vrho_{u_\del},
 \]
 but  in view of \eqref{subset-del} if $\vrho\leq \vrho_{u_\del},$ $\frac{1}{c}\vrho^{2-\del}(E+\ha \mu e^{-2\vrho})\leq 2,$ only if $\mu$ satisfies
 \[
  \mu \leq \frac{4c}{\vrho^{2-\del}} e^{2\vrho} \leq \frac{4c}{\vrho_{u_\del}^{2-\del}} e^{2\vrho_{u_\del}}= 2E e^{\left(\frac{2c}{E}\right)^{\frac{1}{2-\del}}} \defeq \mu_{{}_{\del,U}}
  \]
then 
\[
 \frac{1}{c}  \mce_1(\vrho)\leq \tha \text{ provided } \vrho\in[ \vrho_0, \vrho_{u_\del}] \text{ and } \mu \leq \mu_{{}_{\del,U}}.
\]

  As usual, see for example Chapter 8 of \cite{CL},  to count the zeros of $w(t)$ for $t\leq t_{u_\del}= 2\frac{\sqrt{c}}{\del} \vrho_{u_\del}^{\frac{\del}{2}},$  one sets
\[
\theta(t)=\tan^{-1}\left( \frac{w(t)}{w'(t)}\right), \text{ where } w(t) \text{ satisfies \eqref{conj-pde-del}. }
\]
The number of zeros of $w$  in the interval $[t_0, t_{u_\del}],$ coincides with the number of times $\theta(t)= k\pi,$ for some $k\in \mn.$  But it follows from \eqref{conj-pde-del} that 
\beq\label{difftheta-del}
\frac{d \theta}{d t} =\frac{ (w')^2- w'' w}{w^2+(w')^2}= 1-\frac{1}{c} \mce_1(t(\vrho)) \left( \frac{ w^2}{w^2+(w')^2}\right),
\eeq
and  since $-1\leq \frac{1}{c}\mce_1(t(\vrho))<\tha,$ it follows that $\left|\frac{d \theta}{d t}\right| \leq 3,$
and so
\beq \label{esttheta0-del}
\theta(t_{u,\del})-\theta(t_0)\leq   3(t_{u_\del}-t_0) \leq 3 t_{u_\del}, \;\ t_{u_\del}=\frac{2\sqrt{c}}{\del} \vrho_{u_\del}^{\frac{\del}{2}}.
\eeq
Therefore if $ Z(w)$ denotes the number of zeros of $w(t),$ then
\[
 Z(w)\leq  \frac{3 t_{u_\del}}{\pi}= \frac{6\sqrt{c}}{\pi \del} \biggl(\frac{2c}{E}\biggr)^{\frac{\del}{2(2-\del)}}.
 \]
So we conclude that for $E$ small
 \[
  Z_{j}^\bullet(E) \leq \frac{6\sqrt{c}}{\pi \del} \left(\frac{2}{cE}\right)^{\frac{\del}{2(2-\del)}}  \text{ for all } j.
 \]
  But  in view of  \eqref{mult} and \eqref{NEV}, as $E\rightarrow 0,$  we have for $\bullet=D,N,$ 
 \[
 \begin{split}
   N_E(\mcm^\bullet ) \leq  \sum_{\zeta_j\leq \mu_{{}_{\del,U}}}  m_j(\zeta_j)  Z_{j}^\bullet (E) \leq  C \left(\frac{2}{cE}\right)^{\frac{\del}{2(2-\del)}} \sum_{\zeta_j\leq \mu_{{}_{\del,U}}} m_j(\zeta_j) = C \left(\frac{2}{cE}\right)^{\frac{\del}{2(2-\del)}} N_{\mu_{{}_{\del,U}}}(\Delta_{h(0)}),
 \end{split}
 \]
  $\mcn_{\ka}(\Delta_{h(0)}),$ is the number of  eigenvalues of  $\Delta_{h(0)}$ which are less than or equal to  $\ka$ counted with multiplicity.  Weyl's Law, see  for example Corollary 17.5.8 of \cite{Hor3},  says that 
 \beq\label{weyl}
 \mcn_{\ka}(\Delta_{h(0)})= C_n \ka^n + O(\ka^{n-1}),
 \eeq
  and this implies that
 \[
  N_E(\mcm^\bullet) \leq    C \left(\frac{2}{cE}\right)^{\frac{\del}{2(2-\del)}} \mu_{{}_{\del,U}}^n=  C \left(\frac{2}{cE}\right)^{\frac{\del}{2(2-\del)}} \left( E e^{\left(\frac{2}{cE}\right)^{\frac{1}{2-\del}}}\right)^n,
  \]
  and we find that
  \[
  \log \log N_E(\mcm^\bullet)\leq \frac{1}{2-\del} \log E^{-1} + O(1), \text{ as } E\rightarrow 0,
  \]
  which is the upper bound of \eqref{subc}.
  
  To obtain a similar lower bound, we  will find $\vrho_1=\vrho_1(E)< \vrho_2(E)=\vrho_2,$ such that $\vrho_1(E)>\vrho_0,$ for $E$ small enough with $\vrho_0$ as in \eqref{R01}, and $\mu_{\del,L}$ such that 
  \[
   \frac{\vrho^{2-\del}}{c}(E+ \ha \mu e^{-2\vrho}) \leq \frac14 \text{ for } \vrho\in [\vrho_1,\vrho_2] \text{ and } \mu \leq \mu_{\del,L} 
   \]
   and this implies that $\frac{1}{c} \mce_1(\vrho)\leq \frac34$ and  so  
   \[
   [\vrho_1,\vrho_2] \subset \{ \vrho \geq \rho_0: \frac{1}{c} \mce_1(\vrho)< 1\},
   \]
   and therefore the number of zeros of $w(t)$ in this interval is less than or equal to the number of zeros of $w(t)$ in $[\vrho_0,\infty).$  But we deduce from  \eqref{difftheta-del} that
   \[
   \theta(t_2)-\theta(t_1) \geq \frac14(t_2-t_1),
   \]
   and so 
   \[
   Z_{j}^\bullet(E)\geq \frac{1}{4\pi}(t_2-t_1), \;\ t_j= \frac{2\sqrt{c}}{\del} \vrho_j^{\frac{\del}{2}}, \;\ j=1,2.
   \]
   
   Notice that, since $E\vrho^{2-\del}$ is an increasing function and $\mu\vrho^{2-\del}e^{-2\vrho}$ is an decreasing function,  we choose $\vrho_2$ such that
 \[
 \frac{E}{c} \vrho_2^{2-\del}= \frac18, \;\ \vrho_2=\left(\frac{c}{8E}\right)^{\frac{1}{2-\del}},
 \]
 and for $\vrho_1=\ha \vrho_2,$ we only pick  $\mu$ such that
 \[
  \frac{ \mu}{2c} \vrho_1^{2-\del}e^{-2\vrho_1}\leq \frac18,
  \]
  but this implies that 
 \[
 \mu\leq \frac{c}{4 \vrho_1^{2-\del}} e^{2\vrho_1} \leq \frac{c}{4 \vrho_2^{2-\del}} e^{2\vrho_2}=
2E  e^{\left(\frac{c}{8E}\right)^{\frac{1}{2-\del}}} \defeq \mu_{{}_{\del,L}}.
 \]
 Therefore, for this choice of $\vrho_1$ and $\vrho_2,$ and $\mu\leq \mu_{{}_{\del,L}}$ it follows that  $\frac{1}{c}\vrho^{2-\del}(\ha \mu e^{-2\vrho}+E) \leq \oq$ in $[\ha \vrho_2,\vrho_2].$ Therefore, 
 \[
 Z_{j}^\bullet(E) \geq \frac{1}{4\pi}(t_2-t_1)= \frac{2\sqrt{c}}{\pi\del}\left(\vrho_2^{\frac{\del}{2}}- \vrho_1^{\frac{\del}{2}}\right)\geq \frac{2\sqrt{c}}{\pi \del}\vrho_2^{\frac{\del}{2}}\left(1- (\ha)^{\frac{\del}{2}}\right) \geq C\left( \frac{c}{8E}\right)^{\frac{\del}{2-\del}}
 \]
  
  It follows from \eqref{mult} and \eqref{NEV} that as $E\rightarrow 0,$
  \[
  N_E(\mcm^\bullet) \geq \sum_{\zeta_j \leq \mu_{{}_{\del,L}}}  m_j(\zeta_j)  Z_{j}^\bullet (E)  \geq C\left( \frac{c}{4E}\right)^{\frac{1}{2-\del}} (\mu_{{}_{\del,L}})^n\geq
   C\left( \frac{c}{8E}\right)^{\frac{\del}{2-\del}} E^n e^{ n\left( \frac{c}{8E}\right)^{\frac{1}{2-\del}}},
  \]
  and this shows that
  \[
  \log \log N_E(\mcm^\bullet) \geq \frac{1}{2-\del} \log E^{-1} + O(1) \text{ as } E\rightarrow 0.
  \]
  This ends the proof of item T.\ref{itT1} of Proposition \ref{asym-M} and  together with equations \eqref{count-rel}, \eqref{eigen-est1} and \eqref{spec-X0} it also ends the proof of Theorem \ref{main2}.

\subsection{The Proof of Item T.\ref{itT2} of Proposition \ref{asym-M}}
 In this case $V_0(\vrho)= c \vrho^{-2},$ $V_1(\vrho)= \vrho^{-2} (\log\vrho)^{-\eps},$ and  with  $\mcr,$  $\mcp$ given by \eqref{mcr-mcp}, equation \eqref{pde-Rd1} is
\beq\label{pde-R}
\begin{split}
 \biggl( -\frac{d^2 }{d\vrho^2} -c & \vrho^{-2}+ \mcr(\vrho)+\mcp(\vrho) \biggr) u=0, \;\ E>0,\\
 &  u(\vrho_0)=0.
 \end{split}
\eeq
Next we multiply the equation by $\vrho^{2},$ set $u= \vrho^{\ha } w$ and notice that 

\beq\label{conjid1}
\vrho^{-\ha}(\vrho^{2} \frac{d}{d\vrho} )^2  \vrho^{\ha} = (\vrho\frac{d}{d\vrho})^2-\oq, 
\eeq
 then \eqref{pde-R} becomes

\beq\label{conj-pde-R}
\begin{split}
\biggl( -(\vrho \frac{d}{d\vrho})^2   -(c-\oq)  &\, +  \mce(\vrho) \biggr)w=0, \;\ \mce(\vrho)= \vrho^2(\mcr(\vrho)+\mcp(\vrho))\\
& w(\vrho_0)=0.
\end{split}
\eeq

Set $s= \log \vrho,$ and \eqref{conj-pde-R} becomes
\beq\label{cloq}
\begin{split}
\left(-\frac{d^2}{ds^2} -(c-\oq)+\mce(\vrho(s)) \right) w=0, \\
w(s_0)=0, \text{ where }  s_0= \log(\vrho_0).
\end{split}
\eeq
The case $c<\oq:$  Since $\mcp(\vrho)>0,$ and $\vrho^2\mcr(\vrho)\rightarrow 0$ as $\vrho\rightarrow \infty,$ there exists $R>0$ independent of $\mu$ and $E$ such that 
\[
(\oq-c)+ \mce(\vrho)\geq (\oq-c) + \vrho^2 \mcr(\vrho)>0, \text{ for } \vrho_0>R.
\]
Therefore, we have an equation as \eqref{DN-cond},  with $U(s)=(\oq-c) +  \mce(\vrho(s))>0$
  and so   $w$ has no zeros, and again we conclude from \eqref{NEV} that for this choice of $\vrho_0,$ \eqref{zeroE} holds.

The case $c>\oq:$  We set $t=\la s,$ $\la=(c-\oq)^\ha$ and \eqref{cloq} becomes
 \beq\label{cloq1}
\begin{split}
\left(-\frac{d^2}{dt^2} -1+\frac{1}{\la^2} \mce(\vrho(t)) \right) w=0, \\
w(t_0)=0, \text{ where }  t_0= \la\log(\vrho_0),
\end{split}
\eeq
 
The argument used above shows that  the  zeros of the solution $w$ of \eqref{cloq} lie on the set 
\[
\biggl\{ t> t_0: \frac{1}{\la^2} \mce(\vrho(t))< 1\biggr\}
\]
 and as in the first case, we prove upper and lower bounds for the number of zeros of the solution $w(t)$ of \eqref{cloq1} and we start by picking  $\vrho_0$ large so that $|\vrho^2\mcr(\vrho)|<\ha,$ for $\vrho>\vrho_0.$  In this case,
\[
\begin{split}
 \biggl\{\vrho>\vrho_0: & \frac{1}{\la^2} \mce(\vrho) < 1\biggr\}  \subset 
\biggl\{\vrho\geq \vrho_0:  \frac{\vrho^2}{\la^2}(\mcr(\vrho)+\mcp(\vrho) )\leq \tha \biggr\} \\
& \subset \biggl\{\vrho\geq \vrho_0: \frac{\vrho^2}{\la^2}(E+\ha \mu e^{-2\vrho})\leq 2 \biggr\}.
\end{split}
\]

 But we have
\begin{lemma}\label{convex} Let $f(\vrho)>0$ be a $C^\infty$ function such that $f''(\vrho)>0$ and $f(\vrho) >f'(\vrho)$ for 
$\vrho>\vrho_0.$ Then  the function $F(\vrho)= f(\vrho)(E+ \mu e^{-2\vrho})$   is convex and therefore the set 
\[
\Omega(E,\mu,C)=\bigl\{t>t_0: f(\vrho)(E+ \ha \mu e^{-2\vrho})  \leq C\bigr\}
\]
is either empty  or is equal to an interval $[ a,b ]$ with $a \geq \vrho_0.$
\end{lemma}
 \begin{proof}   We find that
 \[
  F''(\vrho)= f''(\vrho)(E+ \ha \mu e^{-2\vrho}) + 4\mu(f(\vrho)-f'(\vrho))e^{-2\vrho}>0.
 \]
 \end{proof}
Therefore
\beq\label{inc-0}
 \biggl\{\vrho \geq \vrho_0: \frac{\vrho^2}{\la^2}(E+ \ha \mu e^{-2\vrho}) \leq 2\biggr\}= [a,b] \subset 
 \biggl\{ \vrho\geq \vrho_0: \frac{E}{\la^2} \vrho^2 \leq 2\biggr\} \cap \biggl\{ \vrho\geq \vrho_0: \frac{\mu}{2\la^2} \vrho^2 e^{-2\vrho} \leq 2\biggr\}.
\eeq
 
 If we take $\vrho_u$ to satisfy
\[
\frac{E}{\la^2} \vrho_u^2= 2,
\]
then $b\leq \vrho_u$ and  $\frac{E}{\la_2^2} \vrho^2\leq 2,$ for $\vrho\leq \vrho_u.$  We count the zeros of $w(t)$ in the interval $[t_0, t_u]$ and as above one sets
\[
\theta(t)=\tan^{-1}\left( \frac{w(t)}{w'(t)}\right), \text{ where } w(t) \text{ satisfies \eqref{cloq1} }
\]
and then \eqref{difftheta-del} holds, and since $-1\leq \mce(\vrho(t))\leq  2,$ it follows that $\left|\frac{d\theta}{d \vrho}\right| \leq 3$ and so
\beq \label{esttheta0}
\theta(t_u)-\theta(t_0)\leq  3t_u.
\eeq
Therefore if $ Z(w)$ denotes the number of zeros of $w(\vrho(t))$ in an interval $[t_0, t_u],$ we have
 \[
 Z(w)= \frac{2}{\pi} ( \theta(t_u)-\theta(t_0))\leq \frac{3 t_u}{\pi}.
 \]

 So we conclude that for $E$ small
 \[
  Z_j^\bullet(E) \leq \frac{\la}{\pi}\log(\sqrt{\frac{\la^2}{E}})\leq C \log E^{-\ha} \text{ as } E\rightarrow 0
 \]
 
 But on the other hand, according to \eqref{inc-0} we must also have
\[
\mu \leq \frac{4\la^2}{ \vrho^2} e^{2\vrho}  \leq \frac{4\la^2 }{ \vrho_u^2} e^{2\vrho_u}= 2E e^{2\sqrt{\frac{2\la^2}{E}}}\defeq \mu_{{}_u},
\]
and so  $\frac{\vrho^2}{\la^2}(E+\ha e^{-2\vrho}\mu\} \leq 2$ on the interval $[\vrho_0,\vrho_u]$ and $\mu\leq \mu_u,$ 
and in view of \eqref{mult} and  \eqref{NEV}, as $E\rightarrow 0,$  we have
 \[
 \begin{split}
  N_E(\mcm^\bullet) \leq  \sum_{\zeta_j \leq \mu_{{}_U}}  m_j(\zeta_j)  Z_j^\bullet (E) \leq  C \log E^{-\ha} \sum_{\zeta_j \leq \mu_{{}_U}} m_j(\zeta_j) = C \log E^{-\ha}  \mcn_{\mu_{{}_u}}(\Delta_{h(0)}),
 \end{split}
 \]
 and because of \eqref{weyl} this implies that
 \[
  N_E(\mcm^\bullet) \leq   C {\mu^n_{{}_u}} \log E^{-\ha}.
  \]
  It follows from the definition of $\mu_{{}_u}$  that
  \[
 \log \log N_E(\mcm^\bullet)\leq  -\ha \log E + O(1)
 \]
  which gives the upper bound in \eqref{bdhs1}.  
  
  To prove a similar lower bound for $N_E(\mcm^\bullet)$ will find $\vrho_1=\vrho_1(E)<\vrho_2(E)=\vrho_2$ and  $\mu_L$ such that $\vrho_0<\vrho_1(E)$ for $E$ small enough and 
  \[
  \frac{\vrho^2}{\la^2} (E+\ha \mu e^{-2\vrho}) \leq \oq \text{ for } \vrho \in [\vrho_1,\vrho_2] \text{ and } \mu \leq \mu_L,
  \]
  and so  $\frac{1}{\la^2}\mce(\vrho)< \frac34$ for $\vrho\in [\vrho_1,\vrho_2]$ and this implies that
  \[
 [ \vrho_1,\vrho_2]\subset \biggl\{ \rho\geq \rho_0: \frac{1}{\la^2}\mce(\vrho)<1\biggr\}.
 \]
 It then follows from \eqref{difftheta-del} that for $t\in [t_0,t_1],$ $\frac{d\theta}{d t} \geq \oq$ and so $\theta(t_2)-\theta(t_1) \geq \oq(t_2-t_1)$ and
  \[
  Z(\omega) \geq \frac{1}{4\pi}(t_2-t_1), \;\ t_j=\la \log \vrho_j, \;\ j=1,2.
  \]

 Since $E\vrho^2$ is increasing we pick  $\vrho_2$ such that

\beq\label{R1R2}
 \frac{E}{\la^2} \vrho_2^2= \frac18,
\eeq
and so $ \frac{E}{\la^2} \vrho^2\leq \frac18,$ for $\vrho\leq \vrho_2.$ Pick $\vrho_1$ such that
$\vrho_1=\frac{1}{M} \vrho_2,$  with  $M$  to be chosen, and  we want
\[
\frac{\mu}{2\la^2}\vrho_1^2 e^{-2\vrho_1} \leq \frac18,
\]
thus
\beq\label{muS}
\mu \leq  \frac{\la^2}{4 \vrho_1^2} e^{2\vrho_1}\leq \frac{\la^2}{4 \vrho_2^2} e^{2\vrho_2}= 2E e^{2\sqrt{\frac{\la^2}{8E}}}
 \defeq \mu_{{}_L}.
\eeq
Since $t_1=\la \log \vrho_1$ and $t_2=\la\log \vrho_2,$ we can choose $M,$  independently of $E$ and $\mu,$ such that
  \[
 Z(w)\geq  \frac{1}{4\pi} ( t_2-t_1) \geq  \frac{\la}{4\pi}\log\left(\frac{\vrho_2}{\vrho_1}\right)=
 \frac{\la}{4\pi}\log M>1.
 \]
 
Again we conclude from \eqref{mult}, \eqref{NEV} that for $\mu_{{}_L}$ as in \eqref{muS},
 \[
 \begin{split}
  N_E(\mcm^\bullet) \geq  \sum_{\zeta_j \leq \mu_{{}_L}} m_j(\mu_j)  Z_j^\bullet (E) \geq C  \sum_{\zeta_j\leq \mu_{{}_L} }m_j(\mu_j)\geq  C\mu_{{}_L}^n.
 \end{split}
 \]
 This gives that
 \[
  \log \log N_E(\mcm^\bullet)\geq -\ha \log E+ O(1),
   \]
  which is the lower bound in \eqref{bdhs1} and proves  item T.\ref{itT2} of Proposition \ref{asym-M} and  together with equations \eqref{count-rel}, \eqref{eigen-est1} and \eqref{spec-X0} it also ends the proof of Theorem \ref{main1}.
 
 \subsection{The Proof of Item T.\ref{itT3} of Proposition  \ref{asym-M}}

We first consider the case $N=1$  in \eqref{def-GN} and we have  equation \eqref{pde-Rd1} with 

\[
V_0(e^{-\vrho})= \oq\vrho^{-2} + c_1 \vrho^{-2}(\log \vrho)^{-2} \text{ and } V_1(e^{-\vrho})= \vrho^{-2}(\log \vrho)^{-2} (\log \vrho)^{-\eps},
\] 
and  $\mcr,$  $\mcp$ given by \eqref{mcr-mcp}. We multiply the equation by $\vrho^{2},$ set $u= \vrho^{\ha } w$ and use \eqref{conjid1} and we obtain 
 
 \beq\label{conj-pde-R11}
\begin{split}
  \biggl( -(\vrho \p_\vrho)^2 - & c_1(\log \vrho)^{-2} + \vrho^2 \mce(\vrho) \biggr)  w=0, \;\  \mce(\vrho)= \mcr(\vrho)+\mcp(\vrho),  \\
& w(\vrho_0)=0,
\end{split}
\eeq

 Now we set $\xi=\log \vrho,$  multiply \eqref{conj-pde-R11} by $\xi^2$ and set $w=\xi^\ha v,$ use  \eqref{conjid1} and we obtain
\[
\begin{split}
 \biggl( -(\xi \frac{d}{d\xi})^2 & -( c_1-\oq)+  \mce_1(\vrho(\xi)) \biggr) v=0, \;\  \mce_1(\vrho)= \vrho^2(\log \vrho)^2\mce(\vrho), \\
&  v(\xi_0)=0.
\end{split}
\]
As before, if $c_1<\oq,$ $v$ has no zeros for $\xi>\xi_0,$  and so it follows from \eqref{NEV} that $\mcm^\bullet,$ $\bullet=N,D$  have no eigenvalues. 

 If $c_1>\oq,$ set $\la_1=(c_1-\oq)^\ha$ and set $\tau= \la_1\log  \xi,$ and we obtain
\[
\begin{split}
\biggl( - \p_\tau^2 -1 & + \frac{1}{\la_1^2}  \mce_1(\tau(\vrho)) \biggr) v=0,  \;\ \tau=\la_1 \log_{(2)} \vrho,\\
& \quad \quad \quad v(\tau_0)=0.
\end{split}
\]

As before, we assume that $\vrho_0$ is so large  that
\[
\left|\frac{1}{\la_1}^2 \vrho^2(\log \vrho)^2 \mcr(\vrho)\right|\leq \ha, \text{ for } \vrho \geq \vrho_0,
\]
and for this choice of $\vrho_0,$ we have
\beq\label{ineq3}
\begin{split}
& \biggl\{\vrho\geq \vrho_0:  \frac{1}{\la_1^2}\mce_1(\vrho) <1   \biggr\}  \subset   \biggl\{\vrho\geq \vrho_0: \frac{1}{\la_1^2}(\vrho\log \vrho)^2 \mcp(\vrho) \leq 2\biggr\}\subset  \\
& \biggl\{\vrho\geq \vrho_0:  \frac{E}{\la_1^2}(\vrho\log\vrho)^2   \leq 2 \biggr\}  \  \cap \biggl\{\vrho\geq \vrho_0: \frac{\mu}{2\la_1^2}(\vrho\log\vrho)^2 e^{-2\vrho} \leq 2\biggr\}.
\end{split}
\eeq

Since $f(\vrho)=\vrho^2(\log \vrho)^2$ satisfies the hypothesis of Lemma \ref{convex}, we deduce that
\[
\{\vrho\geq \vrho_0: \frac{1}{\la_1^2}(\vrho\log \vrho)^2 (E + \ha \mu e^{-2\vrho}) \leq 2\}= [a,b],
 \]
 but then we must have
 \[
\frac{E}{\la_1^2}(b\log b)^2  \leq 2
\]
and in particular if we take
\[
 \vrho_{u_1} =\frac{2}{\log A_1} A_1 \text{ where }  A_1=\sqrt{\frac{2\la_1^2}{E}},
 \]
 then, for $E$ small, and independently of $\mu,$
 \[
 \vrho_{u_1} \log \vrho_{u_1}= 2A_1( 1+ \frac{\log 2}{\log A_1}- \frac{1}{\log A_1} \log (\log A_1))\geq A_1,
 \]
 and so $b \leq \vrho_{u_1}$ and therefore for $\vrho\leq \vrho_{u_1},$  in view of \eqref{ineq3}, we must have
\beq\label{defm1}
\begin{split}
 \mu\leq  \frac{4 \la_1^2}{(\vrho \log \vrho)^2} e^{2\vrho}\leq \frac{4 \la_1^2}{(\vrho_u \log \vrho_{u_1})^2} e^{2\vrho_{u_1}}  \defeq
 \mu_{{}_{U_1}}.
\end{split}
\eeq

Since $\tau=\la_1 \log \log \vrho,$ we obtain for small $E,$
\begin{gather*}
\theta(\tau_{u_1})-\theta(\tau_0) \leq 3(\tau_{u_1}-\tau_0)\leq 3\tau_{u_1} =3 \la_1 \log \log \vrho_{u_1}\leq 3 \la_1  \log \log \vrho_{u_1} \leq 
C \la_1 \log \log E^{-1},
\end{gather*}
and for $\mu_{{}_{U_1}}$ as in \eqref{defm1}, 
\begin{gather*}
N_E(\mcm^\bullet) \leq C \la_1  \log \log E^{-1}\mcn_{\p X,h_0}(\mu_{{}_{U_1}})\leq C \la_1\log ( \log E^{-1}) ( \mu_{{}_{U_1}})^n.
\end{gather*}
It follows that
\[
\log_{(3)} N_E(H)\leq  \ha \log_{(2)} E^{-1} +O(1),
\]
and this gives the upper bound in \eqref{bdhs2}.   

To obtain the lower bound we find $\vrho_1=\vrho_1(E)< \vrho_2(E)=\vrho_2$ and $\mu_L$ such that $\vrho_0<\vrho_1(E)$ for $E$ small enough and
\[
\frac{\vrho^2(\log \vrho)^2}{\la_1^2}(\ha \mu e^{-2\vrho}+E)\leq \frac14 \text{ for } \vrho\in [\vrho_1,\vrho_2] \text{ and }\mu\leq \mu_L,
\]
and this can be achieved if we take $\vrho_1$ and $\vrho_2$ such that
\[
  \frac{E}{\la_1^2} (\vrho_2\log \vrho_2)^2 \leq \frac{1}{8}, \text{ and }  \vrho_1= \vrho_2^{\frac{1}{M}}, \;\ \text{ with } M \text{ large to be chosen independently of } \mu, E.
\]
For instance, for $E$ small enough, take 
\[
\begin{split}
 \vrho_2= \frac{1}{\log \beta_1} \beta_1, \text{ where } \beta_1=\sqrt{\frac{\la_1^2}{8E}},
\end{split}
\]
and therefore
\[
\vrho_2\log\vrho_2= \beta_1(1-\frac{1}{\log \beta_1} \log_{(2)} \beta_1)< \beta_1.
\]
But we also want $\frac{\mu}{2\la_1^2} \vrho^2(\log \vrho)^2 e^{-2\vrho}\leq \frac18$ and so we need $\mu$ to satisfy

\[
\mu\leq \frac{1}{16(\vrho_1\log \vrho_1)^2}e^{2\vrho_1}\defeq \mu_{{}_{L_1}}.
\]
Since $\tau=\la_1 \log \log \vrho,$ we can choose $M,$  independent of $E$ and $\mu,$ so that 
\begin{gather*}
\theta(\tau_2)-\theta(\tau_1) \geq \frac{1}{4\pi }(\tau_2-\tau_1) \geq  \frac{\la_1}{4\pi}\log\left( \frac{\log \vrho_2}{ \log \vrho_1} \right)=  \frac{\la_1}{4\pi}\log\left( M \right)>1.
\end{gather*}

This implies that for $\mu_{{}_{L_1}}$ as above
\[
N_E(\mcm^\bullet) \geq \mcn_{\p X,h_0}(\mu_{{}_{L_1}})\geq  C (\mu_{{}_{L_1}})^n,
\]
and we conclude that 

\[
\log_{(3)} N_E(\mcm^\bullet)\geq  \log_{(2)} E^{-1} + O(1),
\]
which implies \eqref{bdhs2}.

Next, we consider the case $N=2,$ which corresponds to
\[
\begin{split}
 V_0(\vrho)= &\ \oq \vrho^{-2} + \oq \vrho^{-2}(\log \vrho)^{-2}+ c_2 \vrho^{-2} (\log \vrho)^{-2}(\log_{(2)}\vrho)^{-2}, \\
 V_1(\vrho)= & \  \vrho^{-2} (\log \vrho)^{-2}(\log_{(2)}\vrho)^{-2} (\log \vrho)^{-\eps}, \\
& \mcr(\vrho) \text{ and } \mcp(\vrho) \text{ as in \eqref{mcr-mcp}.}
\end{split}
\]
 This time we set $\eta=\log_{(3)} \vrho$ in equation \eqref{pde-Rd1} and set $u=\bigl(\vrho(\log \vrho) (\log_{(2)} \vrho)\bigr)^\ha v,$ we obtain
\[
\begin{split}
& \left( -(\frac{d}{d\eta})^2 -( c_2-\oq)+  \mce_2(\vrho(\eta)) \right) v=0, \text{ where } \\
& \mce_2(\vrho)= \vrho^2(\log \vrho)^2(\log \log \vrho)^2 (\mcr(\vrho)+\mcp(\vrho)),  \\
& \quad \quad \quad \quad  v(\eta_0)=0.
\end{split}
\]
As before, if $c_2<\oq,$ $v$ has no zeros for $\mu>\mu_0,$ and so $\mcm^\bullet$ has no negative eigenvalues.  If $c_2>\oq,$ set $\la_2=(c_2-\oq)^\ha$ and set $\tau= \la_2 \eta,$ and we obtain
\[
\begin{split}
&\left( - (\frac{d}{d\tau})^2 -1 + \frac{1}{\la_2^2}  \mce_2(\tau(\vrho)) \right) v=0, \\
& \quad \quad \quad v(\tau_0)=0, 
\end{split}
\]
   
   We pick $\vrho_0$ large such that 
   \[
   \left| \frac{1}{\la_2^2} \vrho^2(\log \vrho)^2(\log \log \vrho)^2 \mcr(\vrho)\right| \leq \ha, \text{ if }  \vrho\geq \vrho_0,
   \]
  and for $\mcg_2=(\vrho^2(\log \vrho)^2(\log \log \vrho)^2)^{-1},$ it follows that
\[
\begin{split}
& \biggl\{\vrho\geq \rho_0:  \frac{1}{\la_2^2}  \mce_2(\vrho)<1 \biggr\} \subset \biggl\{\vrho\geq \vrho_0: \frac{1}{\la_2^2 \mcg_2(\vrho)}  (E + \ha \mu e^{-2\vrho}) \leq 2\biggr\}\subset  \\
&  \biggl\{\vrho\geq \vrho_0: \frac{E}{\la_2^2 \mcg_2(\vrho)}  \leq 2\biggr\}\cap \biggl\{\vrho\geq \vrho_0: \frac{\mu}{2\la_2^2\mcg_2(\vrho)} e^{-2\vrho} \leq 2\biggr\}.
\end{split}
\]
and since $f(\vrho)=\frac{1}{\mcg_2(\vrho)}= \vrho^2(\log \vrho)^2(\log \log \vrho)^{2}$ satisfies the hypothesis of Lemma \ref{convex}, we find that
\[
\biggl\{\vrho\geq \vrho_0: \frac{1}{\la_2^2 \mcg_2(\vrho)} (E + \ha \mu e^{-2\vrho}) \leq 2\biggr\}=[a,b ],
\]
and so for $\vrho\in [a,b]$  we must have
\[
\frac{E}{\la_2^2 \mcg_2(\rho)}  \leq 2 \text{ and so } \vrho (\log \vrho) (\log \log \vrho) \leq \sqrt{\frac{2\la_2^2}{E}}=A_2.
\]

If we take $\vrho_{u_2}=\frac{2A_2}{\log A_2(\log \log A_2)},$ then
\[
 \vrho_{u_2} (\log \vrho_{u_2}) (\log \log \vrho_{u_2})= 2A_2(1+ o(1)) \geq A_2, \text{ as } E\rightarrow 0,
 \]
and so $b \leq \vrho_{u_2}.$ Therefore,  for $\vrho \leq \vrho_{u_2}$ we must have
\[
 \mu\leq 4\la_2^2 \mcg_2(\vrho) e^{2\vrho} \leq 4\la_2^2 \mcg_2(\vrho_{u_2}) e^{2\vrho_{u_2}} \defeq\mu_{{u_2}}.
  \]
It follows from \eqref{NEV} that
\[
N_E(\mcm^\bullet) \leq C (\log_{(3)} \vrho_{u_2} ) (\mu_{{u_2}})^n,
\]
which implies that
\beq\label{LLB}
\log_{(3)} N_E(\mcm^\bullet)\leq \log_{(2)}  E^{-1} + O(1) \text{ as } E \rightarrow 0,
\eeq
and this implies the upper bound of \eqref{bdhs4} when $N=2.$

Again, to obtain the lower bound we find $\vrho_1=\vrho_1(E)< \vrho_2(E)=\vrho_2$ and $\mu_L$ such that $\vrho_0<\vrho_1(E)$ for $E$ small enough and 
\[
\frac{1}{\la_2^2 \mcg_2(\vrho)}(E + \ha \mu e^{-2\vrho}) \leq \frac 34 \text{ for } \vrho\in [\vrho_1,\vrho_2] \text{ and }  \mu\leq \mu_L.
\]
We pick $\vrho_2$ such that 
\[
  \frac{E}{\la_2^2} (\vrho_2(\log \vrho_2)(\log \log \vrho_2))^2 \leq \frac{1}{8}.
\]
For instance we can just take
\[
\begin{split}
&  \vrho_2= \frac{1}{(\log \beta_2) (\log \log \beta_2)} \beta_2, \text{ where } \beta_2=\sqrt{\frac{\la_2^2}{8E}}.
\end{split}
\]
We then  and pick $\vrho_1$ such that   $\log \vrho_1=  (\log \vrho_2)^{\frac{1}{M}},$  with $M$ to be chosen, and we need to restrict the values of $\mu$ so that
\[
\mu \leq \frac{\la_2^2}{4} \mcg_2(\vrho_1) e^{2\vrho_1}=\mu_{{}_{L_2}}.
\]

Since $\tau=\la_2 \log_{(3)} \vrho,$ we can choose $M,$ independent of $\mu$ and $E,$ such that
\[
\theta(\tau_2)-\theta(\tau_1) \geq \frac{1}{4\pi }(\tau_2-\tau_1) \geq 
 \frac{C \la_2}{4\pi} \log\left( \frac{\log(\log \vrho_2)}{\log(\log \vrho_1)}\right) \geq
  \frac{C \la_2}{4\pi} \log M>1.
\]
   
This implies that for $\mu_{{}_{L_2}}$ as above
\begin{gather*}
N_E(\mcm^\bullet) \geq \mcn_{\p X,h_0}(\ka_{L2})\geq C (\mu_{{}_{L_2}})^n.
\end{gather*}
This implies that
\[
\log_{(4)} N_E(\mcm^\bullet)\geq  \log_{(3)} E^{-1} + O(1),
\]
which is the lower bound  \eqref{bdhs4} for the case $N=2.$  Because of the choice of $\rho_1$ we get a weaker lower bound than \eqref{LLB}.

The proof in the general case in \eqref{def-GN} follows the same principle.  We pick $\vrho_0$ such that $\log_{(j)} \vrho>1 \text{ for } \vrho\geq \vrho_0,$ and for all  $j\leq N+1.$ If   $\mcg_{(j)}(\vrho)$  defined in \eqref{def-GN} we have
\[
\begin{split}
& V_0(\vrho)= \oq \vrho^{-2} + \oq \sum_{j=1}^{N-1} \mcg_{(j)}j(\vrho)+ c_N \mcg_{(N)}(\vrho), \;\  V_1(\vrho)=  \mcg_{(N)}(\vrho) (\log \vrho)^{-\eps}, \\
& \text{ and with }  \mcr(\vrho) \text{ and }  \mcp(\vrho) \text{ defined in \eqref{mcr-mcp},}
\end{split}
\]

This time we set $\eta=\log_{(N+1)} \vrho$ in \eqref{pde-Rd1} and set $u=\bigl( \mcg_{N}(\vrho)\bigr)^{-\ha}v,$ we obtain
\[
\begin{split}
& \left( -(\frac{d}{d\eta})^2 -( c_N-\oq)+  \mce_N(\vrho(\eta)) \right) v=0, \;\ \mce_N(\vrho)= (\mcg_{(N)}(\vrho))^{-1} (\mcr(\vrho)+\mcp(\vrho)),  \\
& \quad \quad \quad \quad  v(\eta_0)=0.
\end{split}
\]

If $c_N<\oq,$ then $v$ has no zeros for $\eta>\eta_0,$ if $\eta_0$ is large,  and so $\mcm^\bullet$ has no negative eigenvalues.  If $c_N>\oq,$ set $\la_N=(c_N-\oq)^\ha$ and set $\tau= \la_N \eta,$ and we obtain
\[
\begin{split}
&\left( - (\frac{d}{d\tau})^2 -1 + \frac{1}{\la_N^2}  \mce_N(\vrho(\tau)) \right) v=0,  \;\ \tau=\la_N\log_{(N+1)} \vrho,\\
& \quad \quad \quad v(\vrho_0)=0, 
\end{split}
\]
   
 Again, we  follow the steps in the proof of the previous cases and pick $\vrho_0$ such that
 \[
\biggl| \mcg_N(\vrho)^{-2} \mcr(\vrho)\biggr|\leq \ha, \text{ if } \vrho\geq \vrho_0,
\]
and  so the zeros of $v$ will be contained in the set 
\[
\begin{split}
& \biggl\{ \vrho\geq \vrho_0: \frac{1}{\la_N^2} \mce_N(\vrho)<1\biggr\} \subset 
 \biggl\{\vrho\geq \vrho_0: \frac{1}{\la_N^2 (\mcg_N(\vrho))^2} (E + \ha \mu e^{-2\vrho}) \leq 2\biggr\}\subset \\
& \biggl\{\vrho\geq \vrho_0:  \frac{E}{\la_N^2 (\mcg_N(\vrho))^2}  \leq 2\biggr\}\cap \biggl\{\vrho\geq \vrho_0: \frac{\mu}{2\la_N^2(\mcg_N(\vrho))^2}  e^{-2\vrho} \leq 2\biggr\}.
\end{split}
\]
and since $\frac{1}{(\mcg_N(\vrho))^2}$ satisfies the hypothesis of Lemma \ref{convex}, we find that
\[
\{\vrho\geq \vrho_0: \frac{1}{\la_N^2 (\mcg_N(\vrho))^2}(E + \ha \mu e^{-2\vrho}) \leq 2\}=[a,b].
\]
Since 
\[
\begin{split}
&  \frac{1}{\mcg_N(\vrho)}= \vrho(\log \vrho) (\log \log \vrho) \ldots (\log_{(N)}\vrho)  \leq \sqrt{\frac{2\la_N^2}{E}}, \;\ \vrho\leq \vrho_{*}
 \end{split}
\]
and as before, if we define
\[
\vrho_{u_N}= \frac{2A_N}{\log A_N(\log \log A_N) \ldots \log_{(N)} A_N},  \;\ A_N=\sqrt{\frac{2\la_N^2}{E}},
\]
then
\[
\frac{1}{\mcg_N(\vrho_{u_N})}= \vrho_{u_N} (\log \vrho_{u_N}) \ldots \log_{(N)} \vrho_{u_N} \geq A_N,
\]
and therefore, $b \leq  \vrho_{u_N}.$ We must also have $\mu$ such that
\[
\begin{split}
&  \mu  \leq 4\la_N^2 (\mcg_N(\vrho))^2 e^{2\vrho} = 4 {\mcg_N(\vrho_{u_N})}e^{2\vrho_{u_N}} \defeq \mu_{{u_N}}.
\end{split}
\] 
But on the other hand,
\[
\theta(\tau_2)- \theta(\tau_1) \leq \frac{3}{2\pi} \tau_2\leq  \frac{3}{2\pi} \la_{N} \log_{(N+1)} \vrho_{u_N},
\]
and therefore
\[
N_E(\mcm^\bullet)\leq C ( \log_{(N+1)} \vrho_{u_N})( \mu_{u_N})^n,
\]
and this implies that 
\[
\log N_E(\mcm^\bullet)= \vrho_{u_N}\left(2+ \frac{1}{\vrho_{u_N}}\left( \log \mcg_N(\vrho_{u_N})+ \log (C \log_{(N+1)}\vrho_{u_N})\right)\right),
\]
and so
\[
\log_{(2)} N_E(\mcm^\bullet)= \log \vrho_{u_N}(1+O(1)) \text{ as } E\rightarrow 0.
\]
But
\[
\log \vrho_{u_N}= (\log E^{-1})( \ha + O(1)) \text{ as } E\rightarrow 0,
\]
and therefore
\[
\log_{(2)} N_E(\mcm^\bullet)=(\log E^{-1})(C + O(1)) \text{ as } E\rightarrow 0,
\] 
and  this implies that
\beq\label{LLB1}
\log_{(3)} N_E(\mcm^\bullet) \leq \log_{(2)} E^{-1} + O(1),
\eeq
which is  a better bound than \eqref{bdhs4}.

To establish the lower bound in \eqref{bdhs4}  we need to find $\vrho_1=\vrho_1(E)< \vrho_2(E)=\vrho_2$  and $\mu_L$ such that $\vrho_0<\vrho_1(E)$ for $E$ small enough and 
\[
\frac{1}{\la_N^2 \mcg_N(\vrho)}(\ha \mu e^{-2\vrho}+E)\leq \frac 34 \text{ for } \vrho\in [\vrho_1,\vrho_2], \;\ \mu\leq \mu_L.
\]
We pick  $\vrho_2$ such that 
\[
   \frac{E}{\la_N^2 \mcg_N(\vrho_2)} \leq \frac{1}{8},
\]
and we can just take
\[
\begin{split}
 \vrho_2= \frac{\beta_N}{(\log \beta_N) (\log \log \beta_N) \ldots (\log_{(N)} \beta_N)}, \text{ where } \beta_N=\sqrt{\frac{\la_N^2}{8E}},
\end{split}
\]
and we pick $\vrho_1$ such that 
\[
\log_{(N-1)} \vrho_1= \left(\log_{(N-1)} \vrho_2\right)^{\frac{1}{M}} \text{ with } M \text{ large enough, to be chosen independently of } \mu, E,
\]
and only consider the values of $\mu$ such that
\[
\mu\leq \frac{\la_N^2}{4} \mcg_N(\vrho_1)e^{2\vrho_1}=\mu_{{}_{L_N}}.
\]
Since $\tau=\la_N \log_{{}_{(N+1)}} \vrho,$ we can choose $M,$ independently of $E$ and $\mu,$ large enough such that 
\[
\begin{split}
\theta(\vrho_2)-\theta(\vrho_1) \geq \frac{\la_N}{4\pi} \bigl(\log_{(N+1)} \vrho_2-  \log_{(N+1)} \vrho_1\bigr)=  \frac{\la_N}{4\pi}\log\left(\frac{\log_{(N)} \vrho_2}{\log_{(N)}  \vrho_1}\right)=  \frac{\la_N}{4\pi} \log(M)>1
 \end{split}
\]

In view of \eqref{NEV} and \eqref{weyl}, this implies that 
\begin{gather*}
N_E(\mcm^\bullet) \geq C ( \mu_{{}_{L_N}})^n.
\end{gather*}
Then we have
\[
\log N_E(\mcm^\bullet)\geq 2n \vrho_1+ \log(C \mcg_N(\vrho_1))= \vrho_1(2n+ O(1)),
\]
and so
\[
\log_{(2)} N_E(\mcm^\bullet)= \log \vrho_1+ O(1),
\]
and we deduce that
\[
\log_{N+1} N_E(\mcm^\bullet)= \log_{(N)} \vrho_1+ O(1)= \frac{1}{M} \log_{(N)} \vrho_2+ O(1),
\]
and so
\[
\log_{N+2} N_E(\mcm^\bullet)= \log_{(N+1)} \vrho_2+ O(1).
\]
But on the other hand
\[
\log \vrho_2=( \log E^{-1}) ( \ha + O(1)),
\]
and so
\[
\log_{(j)} \vrho_2= \log_{(j)} E^{-1} + O(1), \;\ j\geq 2,
\]
This implies the lower bound  in \eqref{bdhs4} and ends the proof  of item T.\ref{itT3} of Proposition \ref{asym-M}, and  together with equations \eqref{count-rel}, \eqref{eigen-est1} and \eqref{spec-X0} it also ends the proof of Theorem \ref{doubleT}. Notice that, as in the case $N=2,$ because of the choice of $\vrho_1,$ we get a worse lower bound than \eqref{LLB1}.

\appendix

\section{The spectrum of $\mcm_j$}\label{SPEC}

We will show that if  $V_0(\vrho)$ and $V_1(\vrho)$  satisfy the assumptions of either one of the Theorems \ref{main2}, \ref{main1} or \ref{doubleT}, the operators  $\mcm_j^\bullet$ defined in \eqref{EIGj} have no eigenvalues in $[0,\infty),$ $j\in \mn,$ $\bullet=D,N.$  In particular, this implies  that  the operator $\mcm^\bullet$  defined in \eqref{BVP} with boundary conditions $\bullet=D,N.$   has no eigenvalues $E\geq 0.$ If it did, then each $\mcm_j^\bullet$  would have the same eigenvalue.  We prove the following 
\begin{prop}\label{finite-eigenv}  The operators $\mcm^\bullet_j,$ $j\in \mn,$ $\bullet=D,N,$ $j\in \mn,$ defined in \eqref{mj} have no eigenvalues in $[0,\infty).$
\end{prop}
\begin{proof}   If $E$ is an eigenvalue of $\mcm^\bullet_j,$ then there exists $\psi \in L^2((\vrho_0,\infty))$ such that
\beq\label{aux}
\psi''(\vrho)=\left(-E+V_0(e^{-\vrho})+ aV_1(e^{-\vrho})+ e^{-\vrho}  \wt \mcx(\vrho)\right) \psi(\vrho), \;\ \wt \mcx(\vrho)=\mcx(\vrho)+e^{-\vrho} q(\vrho)\zeta_j.
\eeq

Now we appeal to Theorems 2.1 and 2.4 from Section 6.2 of \cite{Olver}, which we state in a single theorem:

\begin{theorem}\label{olver}
In a given finite or infinite interval $(a_1, a_2)$, let $f(x)$ be a positive,  twice continuously differentiable function, $g(\vrho)$ a continuous real or complex function, and 
\[
F(\vrho) = \int[f^{-1/4} (f^{-1/4})'' - g f^{-1/2}]d\vrho.
\]
Then in this interval the differential equations
\begin{subequations}
\beq 
 u''(\vrho) = (f(\vrho) + g(\vrho))u(\vrho), \text{ and } \label{E1}
 \eeq
 \beq
 w''(\vrho) = (-f(\vrho) + g(\vrho))w(\vrho), \label{E2}
\eeq
\end{subequations}
 have twice continuously differentiable solutions which in the case \eqref{E1} are given by 
\beq\label{sols1}
\begin{split}
&u_1(\vrho) = f^{-1/4}(\vrho)e^{\int f^{1/2} d\vrho}(1 + \eps_1(\vrho)),\\
&u_2(x) = f^{-1/4}(\vrho)e^{-\int f^{1/2} d\vrho}(1 + \eps_2(\vrho)),
\end{split}
\eeq
and in the case \eqref{E2} are given by
\beq\label{sols2}
\begin{split}
&w_1(\vrho) = f^{-1/4}(\vrho)e^{i \int f^{1/2} d\vrho}(1 + \eps_1(\vrho)),\\
&w_2(x) = f^{-1/4}(\vrho)e^{-i\int f^{1/2} d\vrho}(1 + \eps_2(\vrho)),
\end{split}
\eeq
such that the error terms $\eps_j(\vrho),$ $j=1,2$ satisfy
\beq\label{sols3}
\begin{split}
& |\eps_1(\vrho)| \leq e^{\ha \mcv_{a_1,\vrho}(F)} - 1 \text{ and }  |\eps_2(\vrho)| \leq e^{\ha \mcv_{\vrho,a_2}(F)} - 1,\\
& \ha f^{-1/2}(\vrho)|\eps_1'(x)| \leq e^{\ha \mcv_{a_1,\vrho}(F)} - 1 \text{ and } \ha f^{-1/2}(\vrho)|\eps_2'(x)| \leq e^{\ha \mcv_{\vrho,a_2}(F)} - 1, 
\end{split}
\eeq
provided $ \mcv_{a_1,\vrho}(F) < \infty$. Here $ \mcv_{\alpha,\beta}(F) $ denotes the total variation of $F$ on the interval $(\alpha,\beta).$
\end{theorem} 

We first show that one cannot have an eigenvalue $E>0.$ We will consider the case of Theorem \ref{main2},  $V_0(e^{-\rho})=c\rho^{-2+\del}$ and $V_1(\rho)= \rho^{-2+\del}(\log \rho)^{-\eps},$  the other cases are very similar. We apply Theorem \ref{olver} with 
\[
-f(\rho)=- E - c\rho^{-2+\del}+ c_1 \rho^{-2+\del}(\log \rho)^{-\eps}  \text{ and } g(\rho)= e^{-\rho} \mcv(\rho)
\]
Then on the interval $[\rho_1,\infty),$ with $\rho_1$ large,
\[
w_1(\rho)=f^{-\oq} e^{i \sqrt{f(\rho)}}(1+ \eps_1(\rho)), \;\ w_2(\rho)=E^{-\oq} e^{-i \sqrt{f(\rho)} }(1+ \eps_2(\rho)).
\]
But 
\[
(f^{-\oq})'' f^{-\oq}= -\frac{c}{4}(3-\del)(2-\del)E^{-\tha}\rho^{-4+\del}(1+ o(1)) \text{ and } |g(\rho)|\leq C e^{-\rho},
\] 
and so 
\[
\mcv_{\rho_1,\rho}(F)(\rho)= \int_{\rho_1}^\rho |F'(s)| ds \leq \int_{\rho_1}^\rho (M_1 e^{-s}+ M_2 E^{-\tha}s^{-4+\del}) ds \leq 
   C_1( e^{-\rho_1}+ e^{-\rho}) + C_2 E^{-\tha}(\rho_1^{-3+\del}+\rho^{-3+\del})
\]
and so 
\[
\eps_1(\rho) \leq e^{\ha \mcv_{\rho_1,\rho}(F)(\rho)}-1\leq C( ( e^{-\rho_1}+ e^{-\rho}) + E^{-\tha}(\rho_1^{-3+\del}+\rho^{-3+\del}))
, \;\ \del\geq 0,
\]
and hence for $\rho_1$ large, $1+\eps_1(\rho)\sim c_1+ c_2 \rho^{-3+\del}$ and therefore $w_1\not \in L^2([\rho_1,\infty).$ 
A similar analysis works to estimate $\mce_2.$ Since there are constants $C_1$ and $C_2$ such that $\psi(E,\rho)= C_1 w_1(\rho)+ C_2 w_2(\rho),$ it follows that $\psi \not\in L^2([\rho_1,\infty))$ and so $\psi$ cannot be an eigenfunction.

When $E=0,$ and $\del>0,$ we apply the same argument with $-f=-c\rho^{-2+\del}$ and we obtain
\[
(f^{-\oq})'' f^{-\oq}=  c_1\rho^{-1-\frac{\del}{2}}
\]
and so we find that for $\rho_1$ large, $1+ \eps_j(\rho)\sim c_1+ c_2\rho^{-\del}$ and so there are no eigenfunctions with $E\geq 0.$

The last case does not quite apply when $\del=0$ and we use an argument as in the proof of Hardy's inequality in \cite{Davies}.  We will prove the following 
\begin{lemma}\label{HE} Suppose $u\in L^2([\rho_0,\infty)),$ $h(\rho)$ is continuous and $h(\rho)= o(1)$ as $\rho\rightarrow \infty$ and
\[
u''(\rho)= \rho^{-2}(-\oq+ h(\rho)) u \text{ in } (\rho_0,\infty), \;\ c>0,
\]
then $u(\rho)=0$ on $[\rho_0,\infty).$
\end{lemma}
\begin{proof}  Since $u\in L^2([\rho_0,\infty)),$ by using the equation and the Cauchy-Schwarz inequality, we find that 
$|u'(\rho)|\leq C \rho^{-\frac{3}{2}}$ and hence $|u(\rho)|\leq C \rho^{-\frac{1}{2}},$ and the equation gives $|u''(\rho)|\leq C \rho^{-\frac{-5}{2}}.$  Therefore, if $\alpha\in (1,2),$ $\la=\ha(\alpha-1)>0$ and $\rho_1>\rho,$
\[
\begin{split}
& \int_{\rho_1}^\infty \rho^\alpha (u'(\rho))^2 d\rho= \int_{\rho_1}^\infty \rho^\alpha ( \rho^{-\la}(\rho^\la u)' -\la \rho^{-1} u)^2 d\rho\geq \\
& \la^2 \int_{\rho_1}^\infty \rho^{\alpha-2} (u(\rho))^2 d\rho - \la\int_{\rho_1}^\infty \left((\rho^\la  u)^2\right)' d\rho
\end{split}
\]
and since $\la<\ha,$ we deduce that for $\alpha\in (1,2),$
\[
\frac{(\alpha-1)^2}{4}  \int_{\rho_1}^\infty \rho^{\alpha-2} (u(\rho))^2 d\rho\leq \int_{\rho_1}^\infty \rho^\alpha (u'(\rho))^2 d\rho.
\]
We apply the same argument to the second derivative, and use that $|u'(\rho)|\leq C\rho^{-\tha},$ then for $\alpha\in (1,4),$
\[
\frac{(\alpha-1)^2}{4}  \int_{\rho_1}^\infty \rho^{\alpha-2} (u'(\rho))^2 d\rho\leq \int_{\rho_1}^\infty \rho^\alpha (u''(\rho))^2 d\rho.
\]
We combine these two estimates and we obtain, for $\alpha \in (3,4),$
\[
\frac{(\alpha-1)^2}{4} \frac{(\alpha-3)^2}{4}  \int_{\rho_1}^\infty \rho^{\alpha-4} (u(\rho))^2 d\rho\leq \int_{\rho_1}^\infty \rho^\alpha (u''(\rho))^2 d\rho.
\]
But the equation implies that
\[
\frac{(\alpha-1)^2}{4} \frac{(\alpha-3)^2}{4}  \int_{\rho_1}^\infty \rho^{\alpha-4} (u(\rho))^2 d\rho\leq \int_{\rho_1}^\infty 
(-\oq + h(\rho))^2  \rho^{\alpha-4} (u(\rho))^2 \ d\rho.
\]
Pick $\rho_1$ large and $\alpha=4-\eps$ with $\eps$ small and this implies that $u(\rho)=0$ on $[\rho_1,\infty).$ Then $u=0$ on $(\rho_0,\infty)$ by uniqueness.
\end{proof}

\end{proof}

\section{The  Proof of Proposition \ref{lmzero}}\label{Zeros}

 We follow the arguments  used in the proof of Theorem XIII.8 of \cite{RS4}.  We have already established that $\sigma_{ess}(\mcm_j^\bullet)=[0,\infty),$ $\bullet=D,N$ and that there are no eigenvalues in the essential spectrum.   Then one needs to prove three lemmas:
 \begin{lemma}\label{lemma0} Let $V(\rho) \in C^\infty(I),$ $I \subset \mr$ open and let $E\in \mr.$ Let $u(\rho,E),$ not identically zero,  satisfy
 \[
 u''(\rho,E)= (V(\rho) -E) u(\rho,E) \text{ on } I.
 \]
   If $a_0=a_0(E_0)\in I$ is such that $u(a_0,E_0)=0.$ Then there exists $\del>0$ and a $C^\infty$ function $a(E)$ defined for $|E-E_0|<\del$  such that $a(E_0)=a_0$ and $u(a(E),E)=0.$
 \end{lemma}
 \begin{proof}  We know from the  existence and uniqueness and stability theorems for odes that $u(\rho,E)$ is a $C^\infty$ function and since $u(\rho,E)$ is not identically zero, if  $u(a_0,E_0)=0,$ then $\p_\rho u(a_0,E_0)\not=0.$  The implicit function theorem then guarantees that there exists a $C^\infty$ function $a(E)$ defined on an interval  $|E-E_0|<\del$ such that $a(E_0)=a_0$ and  $u(a(E),E)=0.$
 \end{proof}
 \begin{lemma}\label{lemma1} As above, let $\bullet=D,N.$  Let $\mcm_j$ be the operators defined in \eqref{EIGj}.  Let $V_0$ and $V_1$ satisfy the hypotheses of either Theorem \ref{main2}, \ref{main1} or \ref{doubleT}. The following statements about $Z_j^\bullet(E)$ are true:
 \begin{enumerate}[1.]
 \item  \label{it1}If $-E<0,$ then $Z_j^\bullet(E)<\infty.$ 
 \item \label{it1a}  If $Z_j^\bullet(E_0)\geq m,$  there exists $\del>0$ so that $Z_j^\bullet(E)\geq m$ for $|E-E_0|<\del.$
 \item \label{it2} $-E_0<-E,$ then $Z_j^\bullet(E)\geq Z_j^\bullet(E_0).$
 \item \label{it3}  If $-E_0$ is an eigenvalue of $\mcm_j^\bullet,$  and $-E_0<-E,$ then $Z_j^\bullet(E)\geq Z_j^\bullet(E_0)+1.$ 
 \item \label{it4}  If $k>j$ and $\mu_k>\mu_j,$ then $Z_j^\bullet(E)\geq Z_k^\bullet(E).$
 \item \label{it5} If $k>j,$ $-E$ is an eigenvalue and $\mu_k>\mu_j,$  then $Z_j^\bullet(E)\geq Z_k^\bullet(E)+1.$
 \end{enumerate}
\end{lemma}
\begin{proof}  We have already shown that  item \ref{it1} is true.   Lemma \ref{lemma0} says that if $\rho_1<\rho_2< \ldots < \rho_{m-1}< \rho_m\in (\rho_0,\infty)$ are such that  $u_j^\bullet(\rho_j,E_0)=0,$ then there exist $\del>0$  and $C^\infty$ functions $r_j(E)$ defined in $|E-E_0|<\del$ such that $r_j(E_0)=\rho_j$ and that $u_j^\bullet(r_j(E),E)=0,$ and therefore  $Z_j^\bullet(E)\geq m.$

To prove  item \ref{it2}, we first consider the Dirichlet problem. This is the standard form of Sturm oscillation theorem. Let $\rho_0< \rho_1 < \ldots < \rho_n$ be the zeros of $u_j^D(\rho, E_0).$ We claim that  $u^D_j(\rho,E)$ has a zero in each of the intervals $(\rho_j, \rho_{j+1}).$ To see that suppose that $u^D_j(\rho,E)$ does not have a zero in this interval.  By possibly multiplying the functions by $-1,$ we may assume that $u^D_j(\rho,E)>0$   and $u^D_j(\rho,E_0)>0$ in $(\rho_j, \rho_{j+1}).$  In this case the $u_j'(\rho_j,E_0)>0$ and $u_j'(\rho_{j+1},E_0)<0.$  Therefore the integral
\[
\begin{split}
& I^D= \int_{\rho_m}^{\rho_{m+1}} \left[ (u^D_j)'(\rho,E_0) u_j^D(\rho,E)-u_j^D(\rho,E_0) (u_j^D)'(\rho,E)\right]' \ d\rho= \\
&(u^D_j)'(\rho_{m+1},E_0) u_j^D(\rho_{m+1},E)- (u^D_j)'(\rho_{m},E_0) u^D_j(\rho_{m},E)\leq 0.
\end{split}
\]
But on the other hand,
\[
\begin{split}
& I^D=\int_{\rho_m}^{\rho_{m+1}} \left[ (u^D_j)''(\rho,E_0) u_j^D(\rho,E)-u_j^D(\rho,E_0) (u_j^D)''(\rho,E)\right] \ d\rho= \\
& (E_0-E)\int_{\rho_m}^{\rho_{m+1}} u_j^D(\rho,E_0) u_j^D(\rho,E) \ d\rho>0.
\end{split}
\]
If $E_0$ is an eigenvalue, we claim that $u_j^D(\rho,E)$ also has a zero in $(\rho_n,\infty).$ To see that we apply the same idea, but now one needs to justify the convergence of the integral from $\rho_n$ to $\infty.$ We appeal again to Theorem \ref{olver}. If we take $f=E_0,$ and $f=E,$ the solutions of \eqref{EIGj}, for $\rho$ large  are of the form 
\[
\begin{split}
& u_j^D(\rho,E_0)= E_0^{-\oq}\bigl( C_1 e^{-\rho \sqrt{E_0}} (1+\eps_1(\rho)) + C_2 e^{\rho \sqrt{E_0}} (1+\eps_2(\rho))\bigr), \\
& u_j^D(\rho,E)= E^{-\oq}\bigl( \wt C_1 e^{-\rho \sqrt{E}} (1+\eps_1(\rho)) + \wt C_2 e^{\rho \sqrt{E}} (1+\eps_2(\rho)) \bigr).
\end{split}
\]
But since $E_0$ is an eigenvalue,  $u_j^D(\rho,E_0)\in L^2((\rho_0,\infty))$ and $C_2=0.$  Since $E_0>E,$ then integrals will involving  terms of the type $e^{\rho (\sqrt{E}-\sqrt{E_0})}O(1)$ which will converge if $E_0>E.$

As for the Neumann problem, the same argument applies with the exception of the interval $(\rho_0,\rho_1).$  In this case, we know from the assumptions made in \eqref{EIGj} that 
\[
(u_j^N)'(\rho_0,E_0)=(u_j^N)'(\rho_0,E)=0, \;\ u_j^N(\rho_0,E_0)=u_j^N(\rho_0,E)=1,
\]
and we also know that $u_j^N(\rho_1,E_0)=0.$ In this case we would have $u_j^N(\rho,E)>0$ and $u_j^N(\rho,E_0)>0$ in $(\rho_0,\rho_1),$ and so we would have
\[
(u_j^N)'(\rho_1,E_0)\leq 0 \text{ and } u_j^N(\rho_1,E)\geq 0
\]
and therefore, the integral
\[
\begin{split}
& I^N= \int_{\rho_0}^{\rho_{1}} \left[ (u^N_j)'(\rho,E_0) u_j^N(\rho,E)-u_j^N(\rho,E_0) (u_j^N)'(\rho,E)\right]' \ d\rho= 
(u_j^N)'(\rho_1,E_0) u_j^N(\rho_1,E) \leq 0.
\end{split}
\]
But as above,
\[
\begin{split}
& I^N= (E_0-E)\int_{\rho_0}^{\rho_{1}} u_j^N(\rho,E_0) u_j^N(\rho,E) \ d\rho>0.
\end{split}
\]

The same argument can be used to show that $Z_j^\bullet(E)\geq Z_k^\bullet(E),$ provided $j>k$ and $\mu_k>\mu_j.$  In this case, we suppose that $\rho_1<\rho_2 \ldots <\rho_n$ are the zeros of $u_k^D(\rho,E),$ and we want to show that $u_j(\rho,D)$ has a zero in $(\rho_m,\rho_{m+1}).$  We assume there are no zeros of $u_j(\rho,E)$ in $(\rho_m,\rho_{m+1})$ and we may assume that $u_k^D(\rho,E)> 0$ and $u_j^D(\rho,E)> 0$ on $\rho_m,\rho_{m+1}$ and
\[
(u_k^D)'(\rho_m,E)>0,  (u_k^D)'(\rho_{m+1},E)<0, \text{ and } u_j^D(\rho_m,E)\geq 0,  u_j^D(\rho_{m+1},E)\geq 0.
\]
Then the integral
\[
\begin{split}
& I^D= \int_{\rho_m}^{\rho_{m+1}} \left[ (u^D_k)'(\rho,E) u_j^D(\rho,E)-u_k^D(\rho,E) (u_j^D)'(\rho,E)\right]' \ d\rho= \\
& (u^D_k)'(\rho_{m+1},E) u_j^D(\rho_{m+1},E)- (u^D_k)'(\rho_{m},E) u_j^D(\rho_{m},E)\leq 0.
\end{split}
\]
But on the other hand
\[
I^D= \int_{\rho_m}^{\rho_{m+1}} (\mu_k-\mu_j) e^{-2\rho} u_k^D(\rho,E) u_j^D(\rho,E) \ d\rho>0.
\]
The same argument works for the Neumann problem and to prove item \ref{it5}.
\end{proof}

 \begin{lemma}\label{lemma2}   Let   $-\la_{j,k}^\bullet,$ $k=1,2,\ldots,$ denote the  eigenvalues of $\mcm_j^\bullet,$ $\bullet=D,N.$ The following are true
 \begin{enumerate}[I.]
 \item \label{ita}The eigenvalues have multiplicity one.
 \item \label{itb} If $E\geq 0,$ $m\in \mn$ and  $Z_j^\bullet(E)\geq m,$ then $-\la_{j,m}< -E.$  In particular 
\[
N_{E}(\mcm_j^\bullet)\geq Z_j^\bullet(E).
\]
\item \label{itc} $Z_j^\bullet(\la_{j,k}^\bullet)=k-1.$
\end{enumerate}
 \end{lemma}
 \begin{proof} The eigenvalues are simple by the uniqueness theorem for ordinary differential equations.  By dividing an eigenfunction $\psi^\bullet(\rho)$ by a constant, one may assume it will satisfy $\psi^D(\rho_0)=0$ and $(\psi^D)'(\rho_0)=1$ or $(\psi^N)'(\rho_0)=0$ and $\psi^N(\rho_0)=1,$ and one cannot have two different solutions with the same Cauchy data.

 We will show that there exist at least $m$ eigenvalues $\la_{j,k}^\bullet$ which are less than $-E.$ Let   $u_j^\bullet(\rho,E)$ be the solution of \eqref{EIGj} and let $\rho_1<\rho_2< \ldots <\rho_{M},$ $M\geq m,$ and  $\rho_0<\rho_1,$  denote its zeros (not equal to  $\rho_0$ in case $\bullet=D$) and let
 \[
 \psi_k^\bullet(\rho)= \left\{\begin{array}{cc}  u_j^\bullet(\rho,E), \text{ if } \rho_{k}\leq \rho \leq \rho_{k+1}, \; k=0, 1, \ldots, M-1 \\
  0 \text{ otherwise} \end{array} \right.
 \]
 Obviously, $\lan \psi_i^\bullet, \psi_k^\bullet \ran=0,$ if $i \not= k.$ Let  $\mcu$ be the $M$-dimensional subspace spanned by $\psi_k^\bullet.$  If $\psi^\bullet=\sum_{j=1}^{m_0} a_k \psi_k^\bullet,$ one can check that
 \[
 \lan \mcm_j^\bullet \psi^\bullet, \psi^\bullet\ran= -E\lan \psi^\bullet,\psi^\bullet\ran,
 \]
 it follows from  the min-max principle, see Theorem XIII.2 of \cite{RS4}, that  $-\la_{j,M}^\bullet\leq - E$ and in particular $\la_{j,m}^\bullet\leq - E.$  But it follows from item \ref{it1a} of Lemma \ref{lemma1} that if $\eps>0$ is small enough,  
 $Z_j^\bullet(E+\eps)\geq m,$ but then we have shown that in fact $-\la_m(\mcm_j^\bullet)\leq-E-\eps<-E.$  This proves item \ref{itb}.

 It is true that $Z_j^\bullet(\la_{j,1}^\bullet)\geq 0.$  Suppose that $Z_j^\bullet(\la_{j,k-1}^\bullet)\geq k-2.$  But it follows from item \ref{it4} of Lemma \ref{lemma1}  that $Z_j^\bullet(\la_{j,k}^\bullet)\geq Z_j^\bullet(\la_{j,k-1}^\bullet)+1\geq k-1.$  On the other hand, notice that if  $Z_j^\bullet(\la_{j,k}^\bullet)>k-1,$ then by item \ref{itb}, $-\la_{j,k}< -\la_{j,k}.$ So we must have $Z_j^\bullet(-\la_{j,k}^\bullet)\leq k-1.$   This proves item \ref{itc}.
 
\end{proof}

Now we can prove \eqref{zer-eigen}.  We know from item \ref{itb} of Lemma \ref{lemma2} that $N_{E}(\mcm_j^\bullet)\geq Z_j^\bullet(E).$   Since  $Z_j^\bullet(E)<\infty$ if $-E<0,$ suppose that  $N_{E}(\mcm_j^\bullet)> Z_j^\bullet(E)=m,$ then by definition this implies that $-\la_{j,m+1}^\bullet \leq -E,$ but then item \ref{it3} of Lemma \ref{lemma1} and item \ref{itc} of Lemma \ref{lemma2} imply that
\[
Z_j^\bullet(E) \geq  Z_j^\bullet(\la^\bullet_{j,m+1})+1=m+1.
\]
This proves \eqref{zer-eigen}.


\end{document}